\documentclass[11pt,a4paper,dvipsnames]{scrartcl}

\usepackage[utf8]{inputenc}
\usepackage[backend=biber, giveninits=true, style=ieee, dashed=false, sorting = nty,citestyle=numeric]{biblatex}

\renewbibmacro{in:}{}
\DeclareFieldFormat{pages}{#1}

\AtEveryBibitem{\clearfield{doi} 
\clearfield{issn}
\clearfield{isbn}
\clearfield{archivePrefix}
\clearfield{month}
\clearfield{day}
\clearfield{pagetotal}
\clearfield{number}}

\usepackage[english]{babel}
\usepackage{amsmath}
\usepackage{amsthm}
\usepackage{amssymb}
\usepackage{geometry}
\usepackage{bm}
\usepackage{microtype}
\usepackage{mathtools}
\usepackage{csquotes}

\usepackage{subcaption}

\usepackage{enumitem}
\setlist[enumerate]{itemsep=0mm,parsep=2mm,topsep=5pt,label=\textit{(\alph*)}}
\setlist[itemize]{itemsep=1mm,parsep=2mm,topsep=6pt}

\usepackage{xcolor}
\newcommand\myshade{85}
\colorlet{myurlcolor}{Aquamarine}

\definecolor{mycitecolor}{HTML}{5590B4}

\usepackage{hyperref}
\hypersetup{
linkcolor  = black,
citecolor  = mycitecolor!85!black, %
urlcolor   = myurlcolor!\myshade!black,
colorlinks = true,
}
\usepackage[capitalise, noabbrev, nameinlink, nosort]{cleveref}
\crefname{equation}{}{}
\crefname{appsec}{Appendix}{Appendices}

\AddToHook{env/lemma/begin}{\crefalias{theorem}{lemma}}
\AddToHook{env/conjecture/begin}{\crefalias{theorem}{conjecture}}
\AddToHook{env/proposition/begin}{\crefalias{theorem}{proposition}}
\AddToHook{env/corollary/begin}{\crefalias{theorem}{corollary}}

\usepackage{graphicx, standalone, caption, wrapfig}
\usepackage{pgf,tikz,pgfplots}
\pgfplotsset{compat=1.17}
\usetikzlibrary{shapes.geometric}
\usetikzlibrary{cd}
\usetikzlibrary{patterns}
\usetikzlibrary{decorations.pathreplacing}
\usepackage{tkz-euclide}

\def\normaledge{1.2}
\definecolor{edgeblack}{rgb}{0.1,0.1,0.1}
\definecolor{vertexblack}{rgb}{0.05,0.05,0.05}

\theoremstyle{definition}
\newtheorem{theorem}{Theorem}[section]

\newtheorem{lemma}[theorem]{Lemma}
\newtheorem*{lemma*}{Lemma}
\newtheorem*{conjecture*}{Conjecture}
\newtheorem*{lemma''*}{``Lemma''}

\newtheorem*{claim*}{Claim}
\newtheorem{corollary}[theorem]{Corollary}

\newtheorem{conjecture}[theorem]{Conjecture}

\newtheorem{example}[theorem]{Example}
\newtheorem{proposition}[theorem]{Proposition}

\DeclareMathOperator{\rk}{rk}

\newcommand{\RR}{\mathbb{R}}
\newcommand{\CC}{\mathbb{C}}

\newcommand{\QQ}{\mathbb{Q}}

\def\mainclass[#1]{$#1$-joined}

\newcommand{\stressnullity}[1]{k_{d, #1}}
\newcommand{\stressnullitydim}[2]{k_{#2,#1}}
\newcommand{\algmat}[1]{\mathcal{M}(#1)}
\newcommand{\stressmatroid}[2]{\mathcal{S}_{#1}(#2)}
\newcommand{\gen}[1]{#1^g}
\newcommand{\affinematroid}[2]{\mathcal{A}_{#1}(#2)}
\newcommand{\globrig}[1]{globally $#1$-rigid}
\newcommand{\rig}[1]{$#1$-rigid}
\newcommand{\strongly}[1]{${#1}$-stress-linked}

\newcommand{\stresskernel}[2]{K_{{#1},{#2}}}
\newcommand{\sharedstresskernel}[1]{K}
\newcommand{\contactlocus}[2]{L_{#1,#2}}
\newcommand{\sym}[1]{\text{Sym}(\CC^{#1 \times #1})}
\newcommand{\affineedge}[2]{A(#1,#2)}

\DeclareMathOperator{\aff}{Affine}
\newcommand{\affine}[2]{\aff(#2)}

\newcommand{\quasigen}{quasi-generic}
\newcommand{\coneG}{G^c}
\newcommand{\bigdim}{D}

\title{Stress-linked pairs of vertices and the generic stress matroid}
\author{Dániel Garamvölgyi\thanks{MTA-ELTE Momentum Matroid Optimization Research Group, Department of
Operations Research, Eötvös Loránd University, and HUN-REN Alfréd Rényi Institute of Mathematics, Budapest, Hungary. e-mail: \texttt{daniel.garamvolgyi@ttk.elte.hu}}}
\date{}
\begin{document}

\maketitle
\begin{abstract}
    Given a graph $G$ and a mapping $p : V(G) \rightarrow \RR^d$, we say that the pair $(G,p)$ is a ($d$-dimensional) realization of $G$. Two realizations $(G,p)$ and $(G,q)$ are equivalent if each of the point pairs corresponding to the edges of $G$ have the same distance under the embeddings $p$ and $q$. A pair of vertices $\{u,v\}$ is globally linked in $G$ in $\RR^d$ if for every generic realization $(G,p)$ and every equivalent realization $(G,q)$, $(G+uv,p)$ and $(G+uv,q)$ are also equivalent. %
    
    In this paper, we introduce and investigate the notion of \strongly{d} vertex pairs. Roughly speaking, a pair of vertices $\{u,v\}$ is \strongly{d} in $G$ if the edge $uv$ is generically stressed in $G+uv$ and for every generic $d$-dimensional realization $(G,p)$, every configuration $q$ that satisfies the equilibrium stresses of $(G,p)$ also satisfies the equilibrium stresses of $(G+uv,p)$. Among other results, we show that \strongly{d} vertex pairs are globally linked in $\RR^d$, and we give a combinatorial characterization of \strongly{2} vertex pairs that matches the conjectural characterization of globally linked pairs in $\RR^2$ due to Jackson et al.
    
    As a key tool, we introduce and study the  ``algebraic dual'' of the $d$-dimensional generic rigidity matroid of a graph $G$, which we call the $d$-dimensional generic stress matroid of $G$. Our results about this matroid, which describes the global behavior of equilibrium stresses of generic realizations of $G$, may be of independent interest.
    
    We use our results to give positive answers to a conjecture of Jordán on minimally globally rigid graphs, a conjecture of Jordán and the author on globally linked vertex pairs, and to conjectures of Connelly and Grasegger et al.\ on rigidity properties of graphs with small separators.  
\end{abstract}

\newpage

\tableofcontents

\newpage

\section{Introduction}

One of the central unsolved problems in combinatorial rigidity theory is finding a combinatorial characterization of \globrig{d} graphs for $d \geq 3$. Roughly speaking, a graph $G$ is \emph{\globrig{d}} if for every generic embedding $p : V(G) \rightarrow \RR^d$ of its vertices into Euclidean space, the (labeled) set of edge lengths $\big\{\lVert p(u) - p(v) \rVert, uv \in E(G)\big\}$ uniquely determines all pairwise distances $\big\{\lVert p(u) - p(v)\rVert, u,v \in V(G)\big\}$. (Precise definitions and references are given in the next section.)
It is folklore that a graph is \globrig{1} if and only if it is $2$-connected, and Jackson and Jordán~\cite{jackson.jordan_2005} gave a combinatorial characterization of \globrig{2} graphs, but the $d \geq 3$ case is still wide open.

However, there is a linear-algebraic characterization of globally $d$-rigid graphs, for all $d$, due to Gortler, Healy and Thurston~\cite{gortler.etal_2010}. 
A $d$-dimensional \emph{realization} of $G$ is a pair $(G,p)$ with $p: V(G) \rightarrow \RR^d$. A vector $\omega \in \RR^{E(G)}$ is an \emph{equilibrium stress}, or simply a \emph{stress}, of $(G,p)$ if it satisfies a certain system of equilibrium conditions determined by $(G,p)$. The main result of~\cite{gortler.etal_2010} states, roughly, that a graph on at least $d+2$ vertices is \globrig{d} if and only if for every generic $d$-dimensional realization $(G,p)$ and every generic equilibrium stress $\omega$ of $(G,p)$, the realizations $(G,q)$ that have $\omega$ as a stress are precisely the affine images of $(G,p)$. 
Although this is not a combinatorial characterization, it has nonetheless proved to be a fundamental tool in the combinatorial approach. 

Apart from the characterization itself, a major contribution of~\cite{gortler.etal_2010} is the introduction of tools from algebraic geometry to rigidity theory. In particular, the authors gave a geometric interpretation of stresses as the tangent hyperplanes of the \emph{measurement variety} of $G$, which is the affine variety of squared edge measurements of $d$-dimensional realizations of $G$. This geometric picture is at the heart of our investigations in this paper.

While there have been a number of important advances in the combinatorial understanding of global rigidity %
since the publication of~\cite{gortler.etal_2010},
a general combinatorial characterization is still elusive. 
One possible avenue for progress 
is to investigate globally linked pairs. A pair of vertices $\{u,v\} \subseteq V(G)$ is \emph{globally linked in $G$ in $\RR^d$} if for every generic $d$-dimensional realization $(G,p)$, the distances $\big\{\lVert p(u') - p(v')\rVert, u'v' \in E(G)\big\}$ uniquely determine the distance $\lVert p(u) - p(v) \rVert$. It is apparent by comparing the definitions that $G$ is \globrig{d} if and only if every pair of vertices is globally linked in $G$ in $\RR^d$. This suggests that characterizing globally linked pairs is at least as hard as characterizing global rigidity, and indeed it seems to be more difficult: an NP-type characterization (combinatorial or otherwise) of globally linked pairs is not even available in the $d=2$ case. %

A natural objective, following~\cite{gortler.etal_2010}, is to find a characterization of globally linked pairs based on equilibrium stresses. An apparent difficulty for this is the fact that, in contrast with global rigidity, being globally linked is not a ``generic property'': it can happen that some generic realizations $(G,p)$ satisfy the condition defining globally linked pairs, while other generic realizations do not. This problem can be rectified by extending our setting to include \emph{complex realizations}, that is, pairs $(G,p)$ where $p : V(G) \rightarrow \CC^d$. %
In this way one can define what it means for a pair of vertices to be globally linked in $G$ in $\CC^d$, and this does turn out to be a generic property. However, this approach presents another difficulty: it seems that the proof method of~\cite{gortler.etal_2010}, which is based on degree theory, cannot be adapted to the complex setting. 

Nonetheless, we 
initiate this line of study 
by introducing a stress-based analog of globally linked vertex pairs in $\RR^d$ (and in $\CC^d$). Roughly speaking, we say that a pair of vertices $\{u,v\} \subseteq V(G)$ is \emph{\strongly{d}} in $G$ if the edge $uv$ is generically stressed in $G+uv$ in $d$ dimensions and for every generic $d$-dimensional realizations $(G,p)$ and $(G,q)$, if $(G,q)$ satisfies the equilibrium stresses of $(G,p)$, then $(G+uv,q)$ satisfies the equilibrium stresses of $(G+uv,p)$. We show that \strongly{d} vertex pairs are globally linked in $\CC^d$, and hence in $\RR^d$ (\cref{theorem:nongloballylinked}). 
We conjecture that the reverse implication also holds, so that being \strongly{d}, globally linked in $\CC^d$, and globally linked in $\RR^d$ are all equivalent (\cref{conjecture:stresskernel}).
As partial evidence for this, 
we show that the combinatorial characterization conjectured  by Jackson, Jordán and Szabadka~\cite{jackson.etal_2006} for globally linked pairs in $\RR^2$, and later by Jackson and Owen~\cite{jackson.owen_2019} for globally linked pairs in $\CC^2$, characterizes \strongly{2} vertex pairs (\cref{theorem:2dimchar}). Among other results, we also settle a ``gluing conjecture'' made by  Jordán and the author in~\cite{garamvolgyi.jordan_2023a} about globally linked pairs in $\RR^d$ by proving that the analogous statement holds for  \strongly{d} pairs (\cref{theorem:gluing}). %

The study of \strongly{d} vertex pairs requires a deeper understanding of the global behavior of stresses of generic frameworks. For example, is it possible that two such stresses (possibly belonging to different realizations) only differ in one coordinate? Or, when can we find such a stress whose values are prescribed on some set of edges? In order to handle such questions we introduce the \emph{dual variety} of the measurement variety, which we call the \emph{($d$-dimensional generic) stress variety} of $G$. While the generic stress variety already appeared implicitly in~\cite{gortler.etal_2010}, we believe that its full significance in rigidity theory is yet to be understood. %

We explore the duality between the measurement variety and the stress variety using the language of matroid theory. It is well-known that to each affine variety $X$ we can associate an \emph{algebraic matroid} $\mathcal{M}(X)$, defined via properties of the coordinate projections of $X$. We call the algebraic matroid associated to the stress variety the \emph{($d$-dimensional generic) stress matroid} of $G$.
The perspective of algebraic matroids highlights a connection that plays an important role in this paper: that the algebraic matroid of a complex affine variety coincides with the linear matroid of its tangent spaces at generic points. For the measurement variety, this is equivalent to the folklore fact that a generic realization $(G,p)$ is infinitesimally rigid if and only if there is, up to congruence, a finite number of realizations $(G,q)$ equivalent to $(G,p)$. However, when applied to the stress variety, the same observation leads to a novel insight: a characterization of the stress matroid in terms of the \emph{generic contact loci} of the measurement variety (\cref{theorem:rigiditycontactlocusmatroid}). We use this characterization to illuminate both the combinatorial structure of the stress matroid, and the geometric structure of the generic contact locus.

In \cref{section:applications}, we apply our results on \strongly{d} vertex pairs to problems in rigidity theory. We say that a graph $G$ is \emph{minimally \globrig{d}} if it is \globrig{d} but $G-e$ is not, for every $e \in E(G)$. In~\cite{jordan_2017}, Jordán gave a conjectured upper bound on the number of edges in a minimally \globrig{d} graph. He also posed the stronger conjecture that if $G$ is \globrig{d} but $G-e$ is not, then $e$ is an $\mathcal{R}_{d+1}$-bridge in $G$. The latter conjecture implies that minimally \globrig{d} graphs are not only sparse (in the sense that they have few edges), but everywhere sparse. While in~\cite{garamvolgyi.jordan_2023} Jordán and the author proved the upper bound on the number of edges, the other conjecture remains open. However, we prove that minimally \globrig{d} graphs are indeed everywhere sparse (\cref{corollary:minglobrigid}). More precisely, we show that in a minimally \globrig{d} graph, every set $X$ of at least $d+2$ vertices induces at most $(d+1)|X| - \binom{d+2}{2}$ edges. We also show that this upper bound is strict unless $X$ is a clique of size $d+2$, and give analogous results for vertex- and edge-redundantly globally rigid graphs.

We also settle two conjectures about graphs with small separators. In~\cite{connelly_2011a}, Connelly showed that if $G_1$ and $G_2$ are \globrig{d} graphs with $d+1$ vertices in common, then $(G_1 \cup G_2) - uv$ is \globrig{d} for every edge $uv \in E(G_1) \cap E(G_2)$, and asked whether the analogous result holds for redundantly \rig{d} graphs. In~\cite{grasegger.etal_2022}, Grasegger, Guler, Jackson and Nixon considered a similar question for so-called $\mathcal{R}_d$-circuits: they conjectured that if $G_1$ and $G_2$ are $\mathcal{R}_d$-circuits such that $G_1 \cap G_2$ is a complete graph on $t$ vertices with $2 \leq t \leq d+1$, then $(G_1 \cup G_2) - uv$ is also an $\mathcal{R}_d$-circuit, for every edge $uv \in E(G_1) \cap E(G_2)$. We give affirmative answers to both conjectures (\cref{theorem:connellyquestion,theorem:tsumcircuitstronger}). Both results follow quickly from combining the ``gluing theorem'' (\cref{theorem:gluing}) with another result on stress-linked vertex pairs which we call the ``cocircuit theorem'' (\cref{corollary:bridges}).  %

\section{Preliminaries}\label{section:preliminaries}

The goal of this section is to set some notation and recall relevant background from matroid theory and combinatorial rigidity theory. \cref{subsection:algmatroids} introduces algebraic matroids -- in our setting, matroids defined by coordinate projections of irreducible varieties. 
In \cref{subsection:rigiditymatroid}, we give the basic definitions of rigidity theory from an algebraic geometric point of view.

To avoid some trivialities, we will implicitly assume that every graph considered throughout the paper has a nonempty set of edges. Given a graph $G$, we use $V(G)$ and $E(G)$ to denote the vertex set and edge set of $G$, respectively. Throughout the paper, $d$ will denote a fixed positive integer.

We assume familiarity with the basic notions of matroid theory; the standard reference is~\cite{oxley_2011}. Given a matroid $\mathcal{M}$, we use $\mathcal{M}^*$ to denote its dual matroid.
Suppose that $\mathcal{M}$ has ground set $E$ and rank function $r$. A bipartition $(E_1,E_2)$ of $E$ is a \emph{separation} of $\mathcal{M}$ if $r(E) = r(E_1) + r(E_2)$. A matroid is \emph{connected} if it has no nontrivial separations, or equivalently, if for every pair of elements of $E$, there is a circuit of $\mathcal{M}$ containing them. Every matroid can be uniquely written as the direct sum of connected matroids, which are called the \emph{connected components} of the matroid. Given two matroids $\mathcal{N},\mathcal{M}$ on the same ground set, we write $\mathcal{N} \preceq \mathcal{M}$ to mean that every independent set in $\mathcal{N}$ is also independent in $\mathcal{M}$.

Given a finite set $E$, we use the notation $\CC^E$ for the affine space whose coordinate axes are labeled by the members of $E$. When using this notation, we will always implicitly assume that $E$ is nonempty. 
For  $\omega \in \CC^E$, we use the notation
\[\omega^\perp = \{x \in \CC^E: \sum_{e \in E} \omega_e x_e = 0\},\]
and similarly, for a linear subspace $L \subseteq \CC^E$ we set
\[L^\perp = \bigcap_{x \in L} x^\perp = \{\omega \in \CC^E : L \subseteq \omega^\perp\}.\]
We call $L^\perp$ the \emph{orthogonal complement} of $L$.
For an irreducible affine variety $X \subseteq \CC^E$ and a point $x \in X$, we let $T_x(X)$ denote the tangent space of $X$ at $x$. We say that $X$ is \emph{defined over $\QQ$} if it can be written as the zero locus of some polynomials with coefficients in $\QQ$. We say that a point $x \in X$ is \emph{$X$-generic} if the only polynomials over $\QQ$ that vanish at $x$ are those vanishing at every point of $X$. See \cref{appendix:alggeo} for a detailed overview of the algebraic geometry background used in the paper. 

\subsection{Matroids of varieties}\label{subsection:algmatroids}

In this subsection we define the algebraic matroid of an irreducible affine variety $X$. For a good reference on this topic, see~\cite{rosen.etal_2020}. Oxley~\cite[Section 6.7]{oxley_2011} also discusses algebraic matroids, but from the perspective of algebraic independence in field extensions, rather than using coordinate projections of affine varieties. As discussed in~\cite{rosen.etal_2020}, the two approaches are equivalent when one is working over an algebraically closed field, as we shall do.

Let $E$ be a finite set, and let $X \subseteq \CC^E$ be an irreducible affine variety. For each subset $E_0 \subseteq E$, let us denote by $\pi_{E_0}$ the restriction of the coordinate projection  $\CC^E \rightarrow \CC^{E_0}$ to $X$. Throughout the section, we will use the notation $X_{E_0} = \overline{\pi_{E_0}(X)}$.

To every irreducible variety $X \subseteq \CC^E$ we associate a matroid $\algmat{X}$ on ground set $E$ defined by
\[E_0 \subseteq E \text{ is independent in } \algmat{X} \Longleftrightarrow X_{E_0} = \CC^{E_0}.\] We call $\algmat{X}$ the \emph{algebraic matroid of $X$}, and we say that $X$ is an \emph{algebraic representation} of $\algmat{X}$. It is known that this indeed defines a matroid. The rank of $\algmat{X}$ is $\dim(X)$, and the restriction of $\algmat{X}$ to a subset $E_0 \subseteq E$ is the algebraic matroid associated to the projection $X_{E_0}$. The following characterization of the spanning sets of an algebraic matroid follows from the fiber dimension theorem (\cref{theorem:fiberdimension}, see also \cref{lemma:genericproperties}(b)).

\begin{lemma}\label{lemma:algmatspanning} \emph{(Geometric characterization of spanning sets.)}\newline
    Let $X \subseteq \CC^E$ be an irreducible affine variety. Fix a subset $E_0 \subseteq E$, and let $\pi_{E_0}: X \rightarrow \CC^{E_0}$ denote the projection of $X$ to $\CC^{E_0}$. 
    \begin{enumerate}
        \item $E_0$ is spanning in $\algmat{X}$ if and only if for a nonempty Zariski open subset of points $x \in X$, the fiber of $\pi_{E_0}$ over $\pi_{E_0}(x)$ is finite.
        \item If $X$ is defined over $\QQ$ and $x$ is an $X$-generic point, then $E_0$ is spanning in $\algmat{X}$ if and only if the fiber of $\pi_{E_0}$ over $\pi_{E_0}(x)$ is finite.
    \end{enumerate}
\end{lemma}

The next lemma relates the separations of $\algmat{X}$ to the geometry of $X$. It is an easy consequence of the fact that the Cartesian product of irreducible affine varieties is an irreducible affine variety.
\begin{lemma}\label{lemma:algmatseparation}\cite[Lemma 3.1]{garamvolgyi.etal_2022} \emph{(Geometric characterization of separations.)}\newline
    Let $X \subseteq \CC^E$ be an irreducible affine variety. A bipartition $(E_1,E_2)$ of $E$ is a separation of $\algmat{X}$ if and only if $X = X_{E_1} \times X_{E_2}$, where $X_{E_i} = \overline{\pi_{E_i}(X)}$. In particular, $e \in E$ is a bridge of $\algmat{X}$ if and only if $X = X_{E-e} \times \CC$.
\end{lemma}

We will use the observation that if $X' \subseteq X \subseteq \CC^E$ are irreducible varieties, then $\algmat{X'} \preceq \algmat{X}$. Indeed, $\overline{\pi_{E_0}(X')} = \CC^{E_0}$ clearly implies $\overline{\pi_{E_0}(X)} = \CC^{E_0}$, for every subset $E_0 \subseteq E$.

The following folklore result relates the algebraic matroid of $X$ to the algebraic matroid of its tangent space at a suitable point. This connection plays a key role in this paper.
\begin{theorem}\label{theorem:tangentspacematroid}
    If $X \subseteq \CC^E$ is an irreducible affine variety, then $\algmat{X} = \algmat{T_x(X)}$ holds for a nonempty Zariski open subset of points $x \in X$. If $X$ is defined over $\QQ$, then $\algmat{X} = \algmat{T_x(X)}$ holds for every $X$-generic point $x \in X$.
\end{theorem}
\begin{proof}
    Note that the projection $\pi_{E_0} : X \rightarrow \CC^{E_0}$ is a linear mapping, and hence it coincides with its differential at any point. Let $X_{E_0}$ denote $\overline{\pi_{E_0}(X)}$. It follows from \cref{theorem:bertini} that there is a nonempty Zariski open subset $U$ of points $x \in X$ such that for every subset $E_0 \subseteq E$, the projection $\pi_{E_0} : T_x(X) \rightarrow T_{\pi_{E_0}(x)}(X_{E_0})$ is surjective. By further restricting $U$, we may also assume that $\pi_{E_0}(x)$ is a smooth point of $X_{E_0}$ for all $E_0 \subseteq E$. Since $X_{E_0} = \CC^{E_0}$ if and only if the tangent space of $X_{E_0}$ at some smooth point is all of $\CC^{E_0}$, it follows that for every $x \in U$ we have $\algmat{X} = \algmat{T_x(X)}$. The statement for generic points is proved in the same way by using \cref{lemma:genericproperties}(c) in place of \cref{theorem:bertini}.
\end{proof}

We shall often consider algebraic matroids of linear spaces. Unsurprisingly, such matroids can also be viewed as linear matroids, as described by the following lemma.

\begin{lemma}\label{lemma:linearalgebraicmatroid} \emph{(Algebraic matroids of linear spaces.)} 
    Let $X \subseteq \CC^E$ be a linear space.
    \begin{enumerate}
        \item  If $M \in \CC^{E \times k}$ is a matrix whose columns span $X$, then $\algmat{X}$ is the linear matroid (over $\CC$) defined by the rows of $M$.
        \item We have $\algmat{X^\perp} = \algmat{X}^*$.%
    \end{enumerate}
\end{lemma}
\begin{proof}
    \textit{(a)} Since $X$ is a linear space, so are its projections. Hence, a subset $E_0 \subseteq E$ is independent in $\algmat{X}$ if and only if the projection of $X$ to $\CC^{E_0}$ is surjective. This projection is the same as the image of the submatrix $M_{E_0}$ consisting of the rows of $M$ corresponding to $E_0$. Thus the projection is surjective if and only if $M_{E_0}$ has full row rank, which is equivalent to $E_0$ being independent in the linear matroid defined by the rows of $M$. \textit{(b)} This now follows from the well-known fact that the linear matroid of a vector space is dual to the linear matroid of the orthogonal complement of the vector space; see, e.g.,~\cite[Section 2.2]{oxley_2011}.
\end{proof}
We will need the following strengthening of \cref{lemma:algmatspanning} for algebraic matroids of linear spaces.
\begin{lemma}\label{lemma:linearspanning}
    Let $X \subseteq \CC^E$ be a linear space, and fix $E_0 \subseteq E$. The following are equivalent.
    \begin{enumerate}
        \item The fiber of $\pi_{E_0}$ over $\pi_{E_0}(x)$ is a singleton for every $x \in X$.
        \item $E_0$ is spanning in $\algmat{X}$.
        \item The fiber of $\pi_{E_0}$  over $\pi_{E_0}(x)$ is a singleton for some $x \in X$.
    \end{enumerate}
\end{lemma}
\begin{proof}
    This follows from \cref{lemma:algmatspanning} and the observation that the projection $\pi_{E_0} : X \to \CC^{E_0}$ is a linear transformation, and hence its nonempty fibers are pairwise isomorphic affine subspaces.
\end{proof}

\subsection{The rigidity matroid}\label{subsection:rigiditymatroid}

In this subsection we recall some basic notions from combinatorial rigidity theory. This exposition is nonstandard for two (related) reasons: we work over the complex field $\CC$ instead of $\RR$, and we mainly work with the algebraic representation of the rigidity matroid given by the measurement variety, rather than the linear representation given by the rigidity matrix of a generic framework. For a general reference on rigidity theory see, for example,~\cite{jordan.whiteley_2017} and~\cite{schulze.whiteley_2017}.

\subsubsection{Basic notions}
Let us fix a graph $G = (V,E)$ for the remainder of this section. Given a mapping $p: V \rightarrow \CC^d$, we say that the pair $(G,p)$ is a \emph{realization} of $G$ in $\CC^d$, or that $(G,p)$ is a \emph{framework} in $\CC^d$. We denote the set of functions from $V$ to $\CC^d$ by $(\CC^{d})^V$, and call the elements of this space \emph{configurations} of $V$. 
A framework $(G,p)$ in $\CC^d$ is \emph{generic} if the coordinates of $p$ do not satisfy any nonzero polynomial with rational coefficients. Using our earlier notation, this means that $p$ is $(\CC^{d})^V$-generic, where we think of $(\CC^{d})^V$ as an affine variety.

For a pair of vertices $u,v \in V$, we let $m_{uv} : (\CC^{d})^V \rightarrow \CC$ denote the squared distance map \[m_{uv}(p) = \sum_{i=1}^d {({p(u)}_i - {p(v)}_i)}^2,\] and we define the \emph{edge measurement map} $m_{d,G} : (\CC^{d})^V \rightarrow \CC^E$ as \[m_{d,G}(p) = {(m_{uv}(p))}_{uv \in E}\] Two realizations $(G,p)$ and $(G,q)$ in $\CC^d$ are \emph{equivalent} if $m_{d,G}(p) = m_{d,G}(q)$. 

The \emph{$d$-dimensional measurement variety} of $G$, denoted by $M_{d,G}$, is the Zariski closure of the image of $m_{d,G}$. This is a homogeneous irreducible affine variety, and by \cref{theorem:constructiblesetsoverQ}, it is defined over $\QQ$. 
The \emph{($d$-dimensional generic) rigidity matroid} of $G$, which we denote by $\mathcal{R}_d(G)$, is the algebraic matroid $\algmat{M_{d,G}}$ associated to $M_{d,G}$.

The \emph{rigidity matrix} of a framework $(G,p)$, denoted by $R(G,p) \in \CC^{E \times dn}$, is the Jacobian of $m_{d,G}$ evaluated at $p$. By combining \cref{theorem:tangentspacematroid}, \cref{lemma:linearalgebraicmatroid}(a), and \cref{lemma:genericproperties}(c), we can deduce the following characterization of $\mathcal{R}_d(G)$, which is commonly taken as its definition.
\begin{proposition}
    For a graph $G$ and any generic framework $(G,p)$ in $\CC^d$, $\mathcal{R}_d(G)$ is the row matroid of $R(G,p)$.
\end{proposition}

We use $r_d(G)$ to denote the rank of $\mathcal{R}_d(G)$, which is equal to the dimension of $M_{d,G}$. We say that $G$ is \emph{$\mathcal{R}_d$-independent} if $r_d(G) = |E(G)|$. Similarly, we say that $G$ is an \emph{$\mathcal{R}_d$-circuit} if it has no isolated vertices and it is not $\mathcal{R}_d$-independent but every proper subgraph of $G$ is $\mathcal{R}_d$-independent. The graph $G$ is \emph{\rig{d}} if $E$ is spanning in $\mathcal{R}_d(K_V)$, the rigidity matroid of the complete graph on $V$. When $n \geq d+1$, this is equivalent to the condition that $r_d(G) = dn - \binom{d+1}{2}$. (If $n \leq d+1$, then $G$ is \rig{d} if and only if it is a complete graph.) An edge $e$ of $G$ is an \emph{$\mathcal{R}_d$-bridge} if $r_d(G) = r_d(G-e) + 1$, or equivalently, if $e$ is not contained in any subgraph of $G$ that is an $\mathcal{R}_d$-circuit. A pair of vertices $\{u,v\}$ is \emph{$\mathcal{R}_d$-linked} in $G$ if $r_d(G + uv) = r_d(G)$. When $u$ and $v$ are nonadjacent in $G$, this means that there is an $\mathcal{R}_d$-circuit in $G+uv$ that contains the edge $uv$, or equivalently that $uv$ is not an $\mathcal{R}_d$-bridge in $G+uv$. Finally, we say that a graph $G$ without isolated vertices is \emph{$\mathcal{R}_d$-connected} if $\mathcal{R}_d(G)$ is a connected matroid. The \emph{$\mathcal{R}_d$-connected components} of $G$ are the subgraphs induced by the connected components of $\mathcal{R}_d(G)$.

\subsubsection{Equilibrium stresses}

Let $(G,p)$ be a framework in $\CC^d$. A vector $\omega \in \CC^E$ is an \emph{equilibrium stress}, or simply a \emph{stress} of $(G,p)$ if $\omega^T R(G,p) = 0$; that is, if $\omega$ lies in the cokernel of $R(G,p)$. We call this cokernel the \emph{space of stresses} of $(G,p)$ and denote it by $S(G,p)$. Note that $\dim(S(G,p)) = |E| - r_d(G)$ for all generic realizations of $G$ in $\CC^d$. In particular, $G$ is an $\mathcal{R}_d$-circuit if and only if every generic realization $(G,p)$ has, up to scale, a single nonzero stress, and this stress is nonzero on all edges of $G$. Similarly, an edge $e \in E$ is an $\mathcal{R}_d$-bridge in $G$ if and only if for every generic realization $(G,p)$, every stress $\omega \in S(G,p)$ satisfies $\omega_e=0$. Finally, a nonadjacent pair of vertices $\{u,v\}$ is $\mathcal{R}_d$-linked in $G$ if and only if every generic realization $(G+uv,p)$ has a stress that is nonzero on $uv$.

The next lemma gives a geometric interpretation of $S(G,p)$ for generic frameworks.
\begin{lemma}\label{lemma:stresstangentspace}(E.g.,~\cite[Lemma 2.21]{gortler.etal_2010})
    Let $(G,p)$ be a framework in $\CC^d$, and let $x = m_{d,G}(p)$. We have $S(G,p) \supseteq {T_x(M_{d,G})}^\perp$, with equality if $(G,p)$ is generic. 
\end{lemma}
\begin{proof}
    The space of stresses is just the cokernel of $R(G,p)$, which by elementary linear algebra is the orthogonal complement of the column space of $R(G,p)$. This column space, being the image of the differential of $m_{d,G}$ at $p$, is contained in $T_x(M_{d,G})$. Moreover, if $(G,p)$ is generic, then by \cref{lemma:genericproperties}(c) the column space is equal to $T_{x}(M_{d,G})$. The statement follows by taking orthogonal complements.
\end{proof}

From \cref{lemma:stresstangentspace} it is easy to deduce the following result, originally proved by Connelly.

\begin{theorem}\label{theorem:connelly} (\cite{connelly_2005})
    Let $(G,p)$ and $(G,q)$ be equivalent frameworks in $\CC^d$. If $(G,p)$ is generic, then $S(G,p) \subseteq S(G,q)$. %
\end{theorem}

When expanded, the condition $\omega^T R(G,p) = 0$ yields the \emph{equilibrium conditions} 
\begin{equation}\label{eq:equilibrium}
    \sum_{uv \in E}\omega_{uv}(p(u) - p(v)) = 0, \hspace{1em} \forall v \in V.
\end{equation}
We can conveniently represent this linear system using the matrix $\Omega \in \CC^{V \times V}$ defined by
\[\Omega_{uv} = \begin{cases}
- \omega_{uv} & \text{if } u \neq v \text{ and } uv \in E \\ 0 & \text{if } u \neq v \text{ and } uv \notin E \\ {\displaystyle \sum_{uw \in E}} \omega_{uw} & \text{if } u = v
\end{cases}\]
We call $\Omega$ the \emph{stress matrix} associated to $\omega$. If we represent the configuration $p$ as a matrix $P \in \CC^{V \times d}$, then $\omega \in S(G,p)$ holds if and only if $\Omega P = 0$. More precisely, the equilibrium conditions \cref{eq:equilibrium} are satisfied for $\omega$ and $(G,p)$ at vertex $v$ if and only if the row of $\Omega P$ corresponding to $v$ is zero. %

Given an arbitrary vector $\omega \in \CC^E$, the \emph{($d$-dimensional) stress kernel} of $\omega$ is the set \[\stresskernel{d}{G}(\omega) = \{p \in (\CC^{d})^V: \omega \text{ is a stress of } (G,p)\}.\] Then $\stresskernel{1}{G}$ is just the kernel of $\Omega$, and since the equilibrium conditions \cref{eq:equilibrium} give an independent system of equations for each coordinate vector, we have $\stresskernel{d}{G}(\omega) = {(\stresskernel{1}{G}(\omega))}^d$. In particular, $\stresskernel{d}{G}(\omega)$ is a linear space. 

Given a framework $(G,p)$ in $\CC^d$, we define the \emph{($d$-dimensional) shared stress kernel} of $(G,p)$ as
\[\sharedstresskernel{d}(G,p) = \{q \in (\CC^d)^V : S(G,p) \subseteq S(G,q)\},\] and similarly, the \emph{one-dimensional shared stress kernel} of $(G,p)$ as
\[K_1(G,p) = \{q \in \CC^V : S(G,p) \subseteq S(G,q)\}.\] 
Formally, we have
\[ \sharedstresskernel{d}(G,p) = \bigcap_{\omega \in S(G,p)} \stresskernel{d}{G}(\omega), \hspace{1em} \text{ and } \hspace{1em} K_1(G,p) = \bigcap_{\omega \in S(G,p)} \stresskernel{1}{G}(\omega).\]
This shows that both $\sharedstresskernel{d}(G,p)$ and $K_1(G,p)$ are linear spaces, and we have $\sharedstresskernel{d}(G,p) = {(K_1(G,p))}^d$.

We say that a configuration $q \in (\CC^d)^V$ is an \emph{affine image} of $p$ if $q(v) = Ap(v) + b$ for all $v \in V$, for some matrix $A \in \CC^{V \times V}$ and vector $b \in \CC^V$. We let $\affine{G}{p} \subseteq (\CC^d)^V$ denote the set of affine images of $p$. This is a linear subspace of $(\CC^d)^V$, and if $(G,p)$ is affinely spanning, then $\dim(\affine{G}{p}) = d(d+1)$. Note that if $\omega \in S(G,p)$ is a stress of $(G,p)$, then the equilibrium conditions \cref{eq:equilibrium} are also satisfied for $\omega$ and $(G,q)$ for every $q \in \affine{G}{p}$. It follows that $\affine{G}{p} \subseteq \stresskernel{d}{G}(\omega)$, and indeed we have the containments 
\begin{equation}\label{eq:stresskernelcontainment}
    \affine{G,p} \subseteq \sharedstresskernel{d}(G,p) \subseteq \stresskernel{d}{G} \subseteq (\CC^d)^V.
\end{equation}

By general principles, the dimension of $K_1(G,p)$ is the same for every generic framework $(G,p)$ of $G$ in $\CC^d$, see~\cite[Remark 4.2]{gortler.etal_2010}.\footnote{In~\cite{gortler.etal_2010}, the authors define $K_1(G,p)$ for real frameworks and stresses. However, their argument also works in the complex setting, and it also shows that the dimension of $K_1(G,p)$ for generic $(G,p)$ is the same in the two settings.} Following~\cite{gortler.etal_2010}, we call this number the \emph{shared stress nullity} of $G$ in $d$ dimensions, and denote it by $\stressnullity{G}$. It follows from \cref{eq:stresskernelcontainment} by taking dimensions and then dividing by $d$ that $d+1 \leq \stressnullity{G} \leq n$ whenever $G$ has at least $d+1$ vertices. 

We will need the following characterization of graphs with large shared stress nullity. The proof is given in \cref{appendix:prelims}.

\begin{lemma}\label{lemma:largenullity} \emph{(Characterization of large shared stress nullity.)}\newline
    Let $G$ be a graph on $n$ vertices.
    \begin{enumerate}
        \item We have $\stressnullity{G} = n$ if and only if $G$ is $\mathcal{R}_d$-independent.
        \item We have $\stressnullity{G} = n-1$ if and only if $G$ consists of a copy of $K_{d+2}$, as well as possibly some $\mathcal{R}_d$-bridges and isolated vertices.
    \end{enumerate}
\end{lemma}

\subsubsection{Globally linked pairs and global rigidity}
A pair of vertices $\{u,v\} \subseteq V$ is \emph{globally linked in $G$ in $\CC^d$} if for every generic framework $(G,p)$ in $\CC^d$ and every framework $(G,q)$ in $\CC^d$, $m_{d,G}(p) = m_{d,G}(q)$ implies $m_{uv}(p) = m_{uv}(q)$. In other words, if $(G,p)$ and $(G,q)$ are equivalent, then so are $(G+uv,p)$ and $(G+uv,q)$.
Similarly, a pair of vertices $\{u,v\} \subseteq V$ is \emph{globally linked in $G$ in $\RR^d$} if for every generic framework $(G,p)$ in $\RR^d$ and every framework $(G,q)$ in $\RR^d$, $m_{d,G}(p) = m_{d,G}(q)$ implies $m_{uv}(p) = m_{uv}(q)$. 

Being globally linked in $G$ in $\CC^d$ clearly implies being globally linked in $G$ in $\RR^d$, and it is unclear whether the two notions are, in fact, equivalent. %
We note that (in contrast with the complex case) it can happen that the condition defining global linkedness in $\RR^d$ holds for some, but not all generic frameworks of $(G,p)$ of $G$ in $\RR^d$. If it holds for some, then $\{u,v\}$ is said to be \emph{weakly globally linked in $G$ in $\RR^d$}; see~\cite{jordan.villanyi_2023} for recent results regarding this notion. 

Interestingly, all these notions of global linkedness are known to be equivalent when we require them to hold for all pairs of vertices.

\begin{theorem}\cite[Lemma 3.2]{jordan.villanyi_2023} and~\cite[Theorem 1]{gortler.thurston_2014a} \emph{(Definitions of globally rigid graphs.)} \newline
    Let $G$ be a graph. The following are equivalent.
    \begin{enumerate}
        \item Every pair of vertices is weakly globally linked in $G$ in $\RR^d$.
        \item Every pair of vertices is globally linked in $G$ in $\RR^d$.
        \item Every pair of vertices is globally linked in $G$ in $\CC^d$.
    \end{enumerate}
\end{theorem}

If any of the above three equivalent conditions hold, then we say that $G$ is \emph{\globrig{d}}. By unraveling condition \textit{(c)}, we see that $G$ is globally $d$-rigid if and only if for every generic realization $(G,p)$ in $\CC^d$, $m_{d,G}(p) = m_{d,G}(q)$ implies $m_{uv}(p) = m_{uv}(q)$ for all $u,v \in V(G)$.

As we noted in the introduction, it is a folklore result that a graph is \globrig{1} if and only if it is $2$-connected. The combinatorial characterization of \globrig{2} graphs is due to Jackson and Jordán. We will need the following form of their result.

\begin{theorem}\label{theorem:2dimglobrigidchar} (\cite{jackson.jordan_2005}, see also~\cite[Theorem 5.1]{jackson.etal_2006}) \emph{(Jackson-Jordán theorem.)} \newline
    Let $G$ be a graph on at least four vertices. The following are equivalent.
    \begin{enumerate}
        \item $G$ is \globrig{2}.
        \item $G$ is $3$-connected and $\mathcal{R}_2$-connected.
        \item $G$ is $3$-connected and $\mathcal{R}_2$-bridgeless.
    \end{enumerate}
\end{theorem}

One of the main results of~\cite{gortler.etal_2010} is the following characterization of \globrig{d} graphs. Originally, it was stated for real frameworks; it was subsequently extended to the complex setting in~\cite{gortler.thurston_2014a}. Part \textit{(b)} is often formulated in terms of the maximum rank of stress matrices, but the following equivalent form will be more convenient for us. %

\begin{theorem}\label{theorem:ght} (\cite{gortler.etal_2010}, see also~\cite[Theorem 1]{gortler.thurston_2014a}) \emph{(Gortler-Healy-Thurston theorem.)}\newline
    Let $G$ be a graph on at least $d+2$ vertices, and let $(G,p)$ be a generic realization in $\CC^d$. The following are equivalent.
    \begin{enumerate}
        \item $G$ is \globrig{d},
        \item there is a stress $\omega \in S(G,p)$ for which $\stresskernel{d}{G}(\omega) = \affine{G}{p}$,
        \item $\sharedstresskernel{d}(G,p) = \affine{G}{p}$,
        \item $\stressnullity{G} = d+1$.
    \end{enumerate}
\end{theorem}

\subsubsection{Small separators and rigidity}
To finish this section, we collect some results about rigidity properties of graphs with small vertex separators.
The first one is a general form of the so-called gluing lemma.
\begin{lemma}\label{lemma:whiteleygluing}\cite[Lemma 11.1.9]{whlong} \emph{(Gluing lemma.)}\newline
    Let $G$ be the union of the graphs $G_1$ and $G_2$.
    \begin{enumerate}
        \item If $|V(G_1) \cap V(G_2)| \geq d$ and $G_1$ and $G_2$ are both \rig{d}, then $G$ is \rig{d}.
        \item If $G_1 \cap G_2$ is \rig{d} and $G_1$ and $G_2$ are both $\mathcal{R}_d$-independent, then $G$ is $\mathcal{R}_d$-independent.
        \item If $|V(G_1) \cap V(G_2)| \leq d-1, u \in V(G_1) - V(G_2)$, and $v \in V(G_2) - V(G_1)$, then $\{u,v\}$ is not $\mathcal{R}_d$-linked in $G$.
    \end{enumerate}
\end{lemma}
The following is a well-known consequence of \cref{lemma:whiteleygluing}(b).

\begin{lemma}\label{lemma:circuitseparatingpair}
    If $C = (V,E)$ is an $\mathcal{R}_d$-circuit, then $C$ is $2$-connected. Furthermore, if $u,v \in V$ is separating vertex pair, then $uv \notin E$.
\end{lemma}

Let $G$ be a graph, and let $u,v \in V(G)$ be a nonadjacent separating pair of vertices. We say that the graphs $G_1,G_2$ are obtained by a \emph{$2$-separation of $G$ along $\{u,v\}$} if $V(G_1) \cap V(G_2) = \{u,v\}$, $uv \in E(G_1) \cap E(G_2)$ and $G = G_1 \cup G_2 - uv$. (The inverse of this, where we glue a pair of graphs $G_1$ and $G_2$ along a pair of edges and then delete the identified edge, is called the \emph{$2$-sum} operation.) 

\begin{lemma}\label{lemma:mconnected2sum}\cite[Lemma 4.9]{garamvolgyi.jordan_2023} and~\cite[Lemma 10]{grasegger.etal_2022} 
    \emph{($\mathcal{R}_d$-connectivity and $2$-separations.)}\newline
    Let $G$ be a graph, let $u,v \in V(G)$ be a separating vertex pair with $uv \notin E(G)$, and let $G_1,G_2$ be obtained by a $2$-separation of $G$ along $\{u,v\}$. The following are equivalent.
    \begin{enumerate}
        \item $G_1$ and $G_2$ are both $\mathcal{R}_d$-connected;
        \item $G$ is $\mathcal{R}_d$-connected;
        \item $G+uv$ is $\mathcal{R}_d$-connected.
    \end{enumerate}
    Moreover, $G$ is an $\mathcal{R}_d$-circuit if and only if $G_1$ and $G_2$ are both $\mathcal{R}_d$-circuits.
\end{lemma}

The following lemma collects some useful observations about the space of stresses of graphs with small separators. The proof is given in \cref{appendix:prelims}.

\begin{lemma}\label{lemma:ksum}
    Let $G$ be the union of the induced subgraphs $G_1 = (V_1,E_1), G_2 = (V_2,E_2)$. Consider a realization $(G,p)$ in $\CC^d$, and let $(G_i,p_i)$ be the subframework corresponding to $G_i$, for $i \in \{1,2\}$. Finally, let $S_i$ denote the subspace of $S(G,p)$ of stresses supported on $E_i$, for $i \in \{1,2\}$.
    \begin{enumerate}
        \item If $G_1 \cap G_2$ is \rig{d} and $(G,p)$ is generic, then $S_1 \cup S_2$ generates $S(G,p)$.
        \item If $G_1 \cap G_2$ is a complete graph on at most $d+1$ vertices, and $(G,p)$ is a framework such that $\{p(v) : v \in V(G_1 \cap G_2)\}$ is affinely independent, then $S(G,p) = S_1 \oplus S_2$.
        \item If $V_1 \cap V_2 = \{u,v\}$, where $\{u,v\}$ is $\mathcal{R}_d$-linked in at most one of $G_1$ and $G_2$, and $(G_i,p_i)$ is generic for both $i \in \{1,2\}$, then $S(G,p) = S_1 \oplus S_2$.
    \end{enumerate}
\end{lemma}

\section{The generic stress matroid}\label{section:genericstressmatroid}

In this section we introduce and study the dual variety of the $d$-dimensional measurement variety, which we call the $d$-dimensional generic stress variety. As the name suggests, the generic stress variety captures the global geometry of equilibrium stresses of generic realizations. We shall be particularly interested in the algebraic matroid associated to this variety, which we call the $d$-dimensional generic stress matroid of $G$. 

In \cref{subsection:algebraicdual} we recall the notion of the dual variety and consider the associated \emph{algebraic dual matroid} in a general setting. While duality of varieties is a classical topic, we are not aware of other works that consider dual varieties from the perspective of algebraic matroids. In \cref{subsection:stressmatroid}, we specialize to the setting of rigidity theory and investigate the algebraic dual matroid of the rigidity matroid (as represented by the measurement variety), and in \cref{subsection:spanningindependent} we obtain some results about the independent sets and the spanning sets of this matroid that will play an important role in our investigation of \strongly{d} vertex pairs in the next section. Finally, in \cref{subsection:globallyrigidstressmatroid} we consider the case of globally rigid graphs in detail. In this case, the dual of the stress matroid has a direct description in terms of the edge measurements of affine images of a generic framework. We use this description to give a complete description of the one-dimensional generic stress matroid of any graph.

We refer the reader to \cref{fig:diag2} for a summary of the different spaces and matroids considered throughout this section.

\subsection{Dual varieties and the algebraic dual matroid}\label{subsection:algebraicdual}

We start by defining the dual variety and stating the \emph{reflexivity theorem}, along with some of its consequences. For a general reference on this topic, see~\cite[Chapter 2]{fischer.piontkowski_2001}, with the caveat that the authors use the language of projective varieties, while we are working with the affine cones of such varieties.

Let $X \subseteq \CC^E$ be an irreducible homogeneous affine variety, and let $X_{sm}$ denote the smooth locus of $X$. The \emph{conormal bundle} $C_X \subseteq \CC^E \times \CC^E$ of $X$ is the Zariski closure
of the set
\[\{(x,\omega) : x \in X_{sm}, \, \omega \in \CC^E, \, T_x(X) \subseteq \omega^\perp\}.\]
This is an irreducible affine variety of dimension $|E|$. Let $\pi_1$ and $\pi_2$ denote the projection of $\CC^E \times \CC^E$ onto the first and second factor, respectively. It is known that for every smooth point $x \in X_{sm}$, 
\begin{equation}\label{eq:conormalprojection}
    \pi_2(\pi_1^{-1}(x)) = {T_x(X)}^{\perp},
\end{equation}
see~\cite[Section 2.1.4]{fischer.piontkowski_2001}.
In other words, taking the Zariski closure in the definition of $C_X$ does not add new points over smooth points of $X$. The \emph{dual variety} of $X$ is defined as $X^* = \overline{\pi_2(C_X)}$. This is an irreducible homogeneous affine variety in $\CC^E$. It can be more concretely described as the Zariski closure of 
\begin{equation}\label{eq:dualdescription}
    \bigcup \left\{{T_x(X)}^\perp: x \in X_{sm}\right\}. 
\end{equation}

The \emph{reflexivity theorem} states that $C_{X^*}$, the conormal bundle of $X^*$, is the same as $C_X$ with the two factors switched:
\[C_{X^*} = \overline{\{(\omega,x) :  x \in X_{sm}, \, \omega \in \CC^E, \, T_x(X) \subseteq \omega^\perp\}}.\]
In particular, this implies that ${(X^*)}^* = X$. 

The following lemma shows that if an affine variety is defined over $\QQ$, then so is its dual. 
We give a proof sketch in \cref{appendix:prelims}.
\begin{lemma}\label{lemma:dualoverQ}
    Suppose that $X$ is defined over $\QQ$. Then we have the following: 
    \begin{enumerate}
        \item $C_X$ and $X^*$ are also defined over $\QQ$,
        \item $X^*$ can be written as the closure of
    \begin{equation}\label{eq:genericclosure}
        \bigcup \left\{{T_x(X)}^\perp : x \in X \text{ is } X\text{-generic}\right\} \vspace{-.3em}
    \end{equation}
    \item     whenever $\omega \in X^*$ is $X^*$-generic, there exists a $C_X$-generic pair $(x,\omega) \in C_X$. For every such pair $(x,\omega)$, $x$ is $X$-generic, $\omega$ is $X^*$-generic, and $\omega \in {T_x(X)}^\perp$. 
    \end{enumerate}
\end{lemma}

Let $X^*_{sm}$ denote the smooth locus of $X^*$. For a point $\omega \in X^*$, the \emph{contact locus} of $\omega$ in $X$ is the set $L_\omega(X) = \pi_1(\pi_2^{-1}(\omega))$. It follows from the reflexivity theorem and \cref{eq:conormalprojection} that if $\omega \in X^*_{sm}$, then $L_\omega(X) = {T_\omega(X^*)}^\perp$. In particular, in this case $L_\omega(X)$ is a linear space of dimension $|E| - \dim(X^*)$; this is sometimes called the \emph{linearity theorem}. The following lemma gives a useful description of the contact locus, under the modest condition that it contains a smooth point of $X$.

\begin{lemma}\label{lemma:contactlocusdecription}
    Fix $\omega \in X^*$. Suppose that $\omega \in X^*_{sm}$ is a smooth point and suppose that $L_\omega(X)$ contains a smooth point of $X$. Then we have
    \[L_\omega(X) = \overline{\{x \in X_{sm}: T_x(X) \subseteq \omega^\perp\}}.\]
    In particular, this holds if $X$ is defined over $\QQ$ and $\omega$ is $X^*$-generic.
\end{lemma}
\begin{proof}
    Since $\omega$ is a smooth point, we have that ${T_\omega(X^*)}^\perp = L_\omega(X)$, so $L_\omega(X)$ is a linear space; in particular, it is irreducible. Since $L_\omega(X)$ contains a smooth point of $X$, it follows that the smooth points are Zariski dense in it. For a smooth point $x \in X_{sm}$ we have \[x \in L_\omega(X) = \pi_1(\pi_2^{-1}(\omega))\] if and only if \[\omega \in \pi_2(\pi_1^{-1}(x)) = {T_x(X)}^\perp,\] or equivalently, if $T_x(X) \subseteq \omega^\perp$. The claim now follows from the density of smooth points. 
\end{proof}

It follows from the fiber dimension theorem that for $X^*$-generic $\omega$, $\dim(X^*) = |E| - \dim(L_\omega(X))$. Since $L_\omega(X) \subseteq X$, we obtain that $|E| \leq \dim(X) + \dim (X^*)$, and this inequality is strict whenever $X$ does not lie in a linear hyperplane of $\CC^E$.

We define the \emph{algebraic dual matroid} of $\algmat{X}$ to be $\algmat{X^*}$. Let us emphasize that the algebraic dual matroid is determined not by the matroid $\algmat{X}$ itself, but rather by its algebraic representation $X$. Also note that, since ${(X^*)}^* = X$, the algebraic dual of $\algmat{X^*}$ is $\algmat{X}$, so this construction is indeed a kind of duality.  However, in general, $\algmat{X}^*$ is not the dual matroid of $\algmat{X}$: indeed, if $X$ does not lie in a linear hyperplane, then 
\[\rk(\algmat{X^*}) = \dim(X^*) \geq |E| - \dim(X) + 1 = |E| - \rk(\algmat{X}) + 1 > \rk(\algmat{X}^*).\]
In fact, the dual of $\algmat{X^*}$ can be described as another algebraic matroid -- the one associated to the contact locus $L_\omega(X)$ for generic $\omega$. This observation plays an important role in this paper.  

\begin{theorem}\label{theorem:contactlocusmatroid}
    Let $X \subseteq \CC^E$ be an irreducible homogeneous affine variety defined over $\QQ$, and let $\omega \in X^*$ be an $X^*$-generic point. Then $\algmat{L_\omega(X)} = \algmat{X^*}^*$. %
\end{theorem}
\begin{proof}
    Recall that $L_\omega(X) = {T_\omega(X^*)}^\perp$. \cref{lemma:linearalgebraicmatroid}(b) implies that $\algmat{L_\omega(X)}$ is the dual matroid of $\algmat{T_{\omega}(X^*)}$, and it follows from \cref{theorem:tangentspacematroid} that the latter is just $\algmat{X^*}$.
\end{proof}

\begin{corollary}\label{corollary:dualrelationships}
    Let $X \subseteq \CC^E$ be an irreducible homogeneous affine variety. We have $\algmat{X}^* \preceq \algmat{X^*}$, and similarly, $\algmat{X^*}^* \preceq \algmat{X}$.
\end{corollary}
\begin{proof}
    We repeat the previous proof for $X$ instead of $X^*$. By \cref{theorem:tangentspacematroid}, there is some smooth point $x \in X$ such that $\algmat{X} = \algmat{T_x(X)}$. By \cref{lemma:linearalgebraicmatroid} we thus have $\algmat{X}^* = \algmat{{T_x(X)}^\perp}$. But from \cref{eq:dualdescription} we have ${T_x(X)}^\perp \subseteq X^*$, and hence $\algmat{{T_x(X)}^\perp} \preceq \algmat{X^*}$. The other statement follows from duality.
\end{proof}

We note that if $X$ is a linear space, then $X^* = X^\perp$, and hence by \cref{lemma:linearalgebraicmatroid}, we have $\algmat{X^*} = \algmat{X^\perp} = \algmat{X}^*$. That is, the algebraic dual matroid does coincide with the dual matroid in the case of linear spaces.

We finish this subsection with a general result about the separations of $\algmat{X^*}$. It is a standard result that for a matroid $\mathcal{M}$ on ground set $E$, a bipartition $(E_1,E_2)$ of $E$ is a separation of $\mathcal{M}$ if and only if it is a separation of the dual matroid $\mathcal{M}^*$. The following proposition is the analogous statement for the algebraic dual matroid. Its proof can be found in \cref{appendix:prelims}.

\begin{proposition}\label{lemma:dualconnectivity}\emph{(Separations of the algebraic dual matroid.)}\newline
    Let $X$ be an irreducible homogeneous affine variety in $\CC^E$, and let $E = E_1 \cup E_2$ be a bipartition of $E$. Then $(E_1,E_2)$ is a separation of $\algmat{X}$ if and only if $(E_1,E_2)$ is a separation of $\algmat{X^*}$. Moreover, if $(E_1,E_2)$ is a separation of $\algmat{X}$, then $X^* = {(X_{E_1})}^* \times {(X_{E_2})}^*$, where $X_{E_i} = \overline{\pi_{E_i}(X)}$ is the projection of $X$ onto $\CC^{E_i}$.
\end{proposition}
\noindent Note that by \cref{lemma:algmatseparation}, if $(E_1,E_2)$ is a separation of $\algmat{X^*}$, then $X^* = {(X^*)}_{E_1} \times {(X^*)}_{E_2}$, where the subscripts again mean that we take the Zariski closure of the projection. However, it is a priori not clear whether ${(X^*)}_{E_i} = {(X_{E_i})}^*$ holds, and indeed, this is not true for a general bipartition $(E_1,E_2)$.

\begin{corollary}\label{corollary:degeneratedual} (\cite[Section 2.2.6]{fischer.piontkowski_2001})
    Let $X$ be an irreducible homogeneous affine variety in $\CC^E$, and let $e \in E$. Then $e$ is a bridge in $\algmat{X}$ if and only if it is a loop in $\algmat{X^*}$.
\end{corollary}
\begin{proof}
    By \cref{lemma:algmatseparation}, we have $X = X_{E-e} \times \CC$. Since the dual of $\CC$ (as a subvariety of $\CC$) is the single point $\{0\}$, \cref{lemma:dualconnectivity} implies that $X^* = X^*_{E-e} \times \{0\}$,  so $e$ is a loop in $\algmat{X^*}$. Conversely, if $e$ is a loop in $\algmat{X^*}$, then the projection of $X^*$ onto the coordinate axis of $e$ is a zero-dimensional homogeneous variety, so it must be the singleton $\{0\}$. Hence we have $X^* = X^*_{E-e} \times \{0\}$, and it follows from \cref{lemma:algmatseparation,lemma:dualconnectivity} again that $e$ is a bridge in $\algmat{X}$.
\end{proof}

\subsection{The generic stress matroid}\label{subsection:stressmatroid}

Let $G$ be a graph and let $(G,p)$ be a framework in $\CC^d$. Recall that $S(G,p)$ denotes the linear space of stresses of $(G,p)$, and that for a stress $\omega \in S(G,p)$, we use  $\stresskernel{d}{G}(\omega) \subseteq (\CC^{d})^V$ to denote the stress kernel of $\omega$, which is the linear space of configurations $q \in (\CC^d)^V$ satisfying $\omega \in S(G,q)$. %

We now specialize the discussion of the previous subsection to the measurement variety $M_{d,G}$.
Recall that for a generic realization $(G,p)$ with $x = m_{d,G}(p)$, we have $S(G,p) = {T_x(M_{d,G})}^\perp$ by \cref{lemma:stresstangentspace}. Also, it follows from \cref{lemma:genericimage} that the $M_{d,G}$-generic points are precisely the images of generic configurations under $m_{d,G}$. Thus by specializing \cref{eq:genericclosure}, we can write the dual variety $M_{d,G}^*$ as the Zariski closure of
\begin{equation}\label{eq:dual}
    \bigcup \left\{S(G,p) : (G,p) \text{ is a generic framework in } \CC^d \right\}.
\end{equation}
Based on \cref{eq:dual}, we shall call $M_{d,G}^*$ the \emph{($d$-dimensional generic) stress variety} of $G$, and $\algmat{M_{d,G}^*}$ the \emph{($d$-dimensional generic) stress matroid} of $G$. Using the terminology of the previous subsection, the $d$-dimensional generic stress matroid, which we will denote by $\stressmatroid{d}{G}$, is the algebraic dual matroid of $\mathcal{R}_d(G) = \algmat{M_{d,G}}$.

Let us say that a vector $\omega \in \CC^E$ is a \emph{generic $d$-stress of $G$} if $\omega$ is $M_{d,G}^*$-generic. By specializing \cref{lemma:dualoverQ} to the measurement variety, we obtain the following basic lemma, which we shall use without reference in the rest of the paper.
\begin{lemma}
    Let $\omega \in \CC^E$ be a generic $d$-stress of the graph $G$. Then $\omega \in S(G,p)$ for some generic framework $(G,p)$ in $\CC^d$. Conversely, every generic framework $(G,p)$ in $\CC^d$ has a generic $d$-stress $\omega \in S(G,p)$.
\end{lemma}
As a weaker notion, we say that the vector $\omega \in \CC^E$ is an \emph{\quasigen\ $d$-stress of $G$} if $\omega \in S(G,p)$ for some generic framework $(G,p)$ in $\CC^d$. 
The description of the stress variety given by \cref{eq:dual} shows that every \quasigen\ $d$-stress is contained in $M_{d,G}^*$. Note that if $G$ is $\mathcal{R}_d$-independent, then $M_{d,G}^* = \{0\}$, and in this case the only \quasigen\ $d$-stress is the zero vector. 

Let $\omega \in \CC^E$ be a vector. The \emph{contact locus} of $\omega$ is the set
\[\contactlocus{d}{G}(\omega) = \overline{\{x \in M_{d,G}: x \text{ is smooth and } T_x(M_{d,G}) \subseteq \omega^\perp\}}.\] By the general results of \cref{subsection:algebraicdual}, whenever $\omega$ is a generic $d$-stress, $\contactlocus{d}{G}(\omega)$ is a linear space, and thus it is irreducible. The following folklore proposition describes how the stress kernel and the contact locus of $\omega$ are related. For completeness, a proof is given in \cref{appendix:prelims}.
\begin{lemma}\label{lemma:contactlocusstresskernel}
    If $\omega$ is a generic $d$-stress of the graph $G$, then $\contactlocus{d}{G}(\omega) = \overline{m_{d,G}(\stresskernel{d}{G}(\omega))}$.
\end{lemma}

For convenience, we state the specialization of \cref{theorem:contactlocusmatroid} to the measurement variety as a separate theorem.
\begin{theorem}\label{theorem:rigiditycontactlocusmatroid}
    Let $\omega$ be a generic $d$-stress of the graph $G$. Then $\algmat{\contactlocus{d}{G}(\omega)} = \stressmatroid{d}{G}^*$.
\end{theorem}

After unpacking the definitions, \cref{theorem:rigiditycontactlocusmatroid} says the following for a set of edges $E_0 \subseteq E$. There exists a \quasigen\ $d$-stress $\omega'$ with $\omega'_e = \sigma_e, e \in E_0$ for almost all $\sigma \in \CC^{E_0}$ if and only if for a generic $d$-stress $\omega$, every point $x \in \contactlocus{d}{G}(\omega)$ in the contact locus is uniquely determined by the values ${(x_e)}_{e \in E - E_0} \in \CC^{E - E_0}$. Conversely, given a generic $d$-stress $\omega$, there exists a point $x \in \contactlocus{d}{G}(\omega)$ with $x_e = y_e$ for all\footnote{Here we do not need to write ``almost all'', since $\contactlocus{d}{G}(\omega)$ is a linear space, and hence its projection is dominant if and only if it is surjective.} $y \in \CC^{E - E_0}$ if and only if for almost all $\sigma \in \CC^{E_0}$, there are at most finitely many \quasigen\ $d$-stresses $\omega'$ with $\omega'_e = \sigma_e, e \in E_0$. These somewhat mysterious correspondences will be useful in clarifying the structures of both the stress matroid $\stressmatroid{d}{G}$ and the contact locus $\contactlocus{d}{G}(\omega)$.

\begin{figure}[tb]
    \centering
    \begin{subfigure}[b]{0.49\linewidth}
    \centering
\begin{tikzcd} %
{(\CC^d)^V} \arrow[r, "m_{d,G}"] & {M_{d,G}} \arrow[r, leftrightarrow, "\text{dual}"]                & {M_{d,G}^*}                              \\
{\stresskernel{d}{G}(\omega)} \arrow[u, "\subseteq", sloped, phantom] \arrow[r, "m_{d,G}"] & {\contactlocus{d}{G}(\omega)} \arrow[u, "\subseteq", sloped, phantom] & {\phantom{\omega \in} S(G,p) \ni \omega} \arrow[u, "\subseteq", sloped, phantom] \\
 {\affine{G}{p}} \arrow[r, "m_{d,G}"] \arrow[u, "\subseteq", sloped, phantom] & {\affineedge{G}{p}} \arrow[u, "\subseteq", sloped, phantom]          &                                         
\end{tikzcd}
\caption{}
\end{subfigure}
\begin{subfigure}[b]{0.49\linewidth}
\centering
\begin{tikzcd} %
\mathcal{R}_d(G) \arrow[r, leftrightarrow, "\text{alg.\ dual}"] \arrow[rd,leftrightarrow] & \stressmatroid{d}{G}                              \\
\stressmatroid{d}{G}^* \arrow[u, "\leq", sloped, phantom] \arrow[ru, leftrightarrow]       & \mathcal{R}_d(G)^* \arrow[u, sloped, "\leq", phantom] \\
\affinematroid{d}{G} \arrow[u, sloped, "\leq", phantom]               &                                              
\end{tikzcd}
\caption{}
\end{subfigure}
    \caption{Relationships between \textit{(a)} some of the spaces considered in this paper, and \textit{(b)} their algebraic matroids. Here $(G,p)$ is a generic realization in $\CC^d$ and $\omega$ is a generic $d$-stress of $G$ with $\omega \in S(G,p)$. In \textit{(b)}, $\mathcal{R}_d(G)$ is the $d$-dimensional generic rigidity matroid of $G$, $\stressmatroid{d}{G}$ is the $d$-dimensional generic stress matroid of $G$, and $\affinematroid{d}{G}$ is the $d$-dimensional generic affine edge measurement matroid of $G$ (introduced in \cref{subsection:globallyrigidstressmatroid}).}
    \label{fig:diag2}
\end{figure}

We close this subsection by considering an example: the complete bipartite graph $K_{5,5}$.

\begin{example}\label{example:K55} \emph{(The $3$-dimensional generic stress matroid of $K_{5,5}$.)} \newline
    Let $U = \{u_1,\ldots,u_5\}$ and $V = \{v_1,\ldots,v_5\}$ denote the bipartition of $V(K_{5,5})$, and let $(K_{5,5},p)$ be a realization of $K_{5,5}$ in $\CC^3$. The sets $p(U)$ and $p(V)$ are affinely dependent, so there exist scalars $a_1,\ldots,a_5 \in \CC$ and $b_1,\ldots,b_5 \in \CC$ with \[\sum_{i=1}^5 a_i = \sum_{i=1}^5 a_ip(u_i) = \sum_{i=1}^5 b_i = \sum_{i=1}^5 b_ip(v_i) = 0.\]
    
    Now the vector $\omega \in \CC^{E(K_{5,5})}$ defined by $\omega_{u_iv_j} = a_ib_j, i,j \in \{1,\ldots,5\}$ is a stress of $(G,p)$.
    Moreover, Bolker and Roth~\cite{bolker.roth_1980} showed that if $p$ is generic, then every stress of $(K_{5,5},p)$ is of this form (see also~\cite[Section 3.4]{gortler.etal_2010}).

    Let us now identify $\CC^{E(K_{5,5})}$ with $\CC^{5 \times 5}$. The previous discussion shows that every \quasigen\ $3$-stress of $K_{5,5}$ is of the form $ab^T \in \CC^{5 \times 5}$ for some $a,b \in \CC^5$ whose coordinates sum to zero. Conversely, it is not difficult to show that almost all such matrices arise as a \quasigen\ $3$-stress of $K_{5,5}$. It follows from \cref{eq:dual} that \begin{align*}
        M_{3,K_{5,5}}^* &= \{ab^T : a,b \in \CC^5, \sum_{i=1}^5 a_i = \sum_{i=1}^5 b_i = 0\} \nonumber \\ & = \{M \in \CC^{5 \times 5}: \rk(M) = 1, M \text{ has zero row and column sums}\}
    \end{align*}
    From this, we can deduce the following:
    \begin{itemize}
        \item The rank of $\mathcal{S}_3(K_{5,5})$ (which is equal to the dimension of $M_{3,K_{5,5}}^*$) is $7$.
        \item Every independent set of $\mathcal{S}_3(K_{5,5})$ induces a forest; in other words,  $\mathcal{S}_3(K_{5,5}) \leq \mathcal{R}_1(K_{5,5})$. While this can be verified directly, it also follows from the fact that $M_{3,K_{5,5}}^*$ is contained in the rank one determinantal variety, whose associated matroid is known to be the graphic matroid of $K_{5,5}$.
        \item Every vertex star is dependent in $\mathcal{S}_3(K_{5,5})$. This corresponds to the fact that every \quasigen\ $3$-stress of $K_{5,5}$ sums to zero on each vertex star.
    \end{itemize}
    It would be interesting to give a direct combinatorial description of $\mathcal{S}_3(K_{5,5})$.
\end{example}

\subsection{Spanning sets and independent sets in the stress matroid}\label{subsection:spanningindependent}

In this subsection we prove a pair of results about the combinatorial structure of the generic stress matroid. These results will play an important role in our investigation of $d$-stress-linked vertex pairs in the next section. The first says that sufficiently large edge sets are always spanning in this matroid, where ``large'' means that the complement of the edge set spans few vertices.  %
\begin{theorem}\label{corollary:stressmatroidspanning} \emph{(Large edge sets are spanning in the stress matroid.)}\newline
    Let $G = (V,E)$ be a graph 
    and let $E_0 \subseteq E$ be a set of edges. If $E - E_0$ spans at most $d+1$ vertices in $G$, then $E_0$ is spanning in $\stressmatroid{d}{G}$.
\end{theorem}

Recall that spanning sets of $\stressmatroid{d}{G}$ correspond to coordinate projections of $\mathcal{S}_d(G)$ with generically finite fibers. Hence
\cref{corollary:stressmatroidspanning} asserts that for a generic $d$-stress $\omega$, the fiber $\pi^{-1}_{E_0}(\omega)$ is finite whenever $E - E_0$ spans at most $d+1$ vertices. This will easily follow from the following lemma. Although its statement is quite natural and the proof is elementary, it appears to be a new result.

\begin{lemma}\label{lemma:stressalmosteveryvertex}
     Let $G = (V,E)$ be a graph, and let $(G,p)$ and $(G,q)$ be realizations in $\CC^d$ such that $(G,p)$ is in general position. Fix $\omega \in S(G,p)$, and let $U \subseteq V$ be the set of vertices where the equilibrium conditions \cref{eq:equilibrium} are not satisfied for $\omega$ and $(G,q)$. If $U\neq \varnothing$, then  $|U| \geq d+2$.
\end{lemma}
\begin{proof}
    We show that if $|U| \leq d+1$, then $U = \varnothing$. Let $\Omega \in \CC^{V \times V}$ be the stress matrix associated to $\omega$, and let $\Omega_U$ denote the submatrix of $\Omega$ formed by deleting the rows corresponding to $U$. Also, let $Q \in \CC^{V \times d}$ be the representation of $q$ as a matrix. Since $(G,q)$ satisfies the equilibrium conditions associated to $\omega$ at the vertices in $U$, we have $\Omega_U Q = 0$. Thus to conclude that $\Omega Q = 0$, it suffices to show that $\Omega_U$ spans the row space of $\Omega$.%
    
    To show this, we use the fact that $\omega$ is the stress of a framework in general position.
    Let $P \in \CC^{V \times d}$ be the representation of $p$ as a matrix, and let $\widetilde{P} \in \CC^{V \times (d+1)}$ be the matrix obtained from $P$ by appending an all-ones column. Note that the rows of $\widetilde{P}$ corresponding to $U$ are linearly independent, since the set $\{p(u), u \in U\}$ is affinely independent. In particular, $\rk(\widetilde{P}) = d+1$.

    Now let us extend $\widetilde{P}$ to a new matrix $P'$ by adding new columns until the column space of $P'$ spans the kernel of $\Omega$. It follows from \cref{lemma:linearalgebraicmatroid} that the row matroid of $P'$ is dual to the column matroid of $\Omega$; since $\Omega$ is symmetric, the latter is equal to the row matroid of $\Omega$. The rows of $P'$ corresponding to $U$ are linearly independent, so by duality, the rows of $\Omega$ corresponding to $V - U$ span the row space of $\Omega$, as claimed. %
\end{proof}

\begin{proof}[Proof of \cref{corollary:stressmatroidspanning}]
    Let us fix a generic $d$-stress $\omega \in M^*_{d,G}$. As noted above, what we need to show is that for the projection $\pi_{E_0} : M^*_{d,G} \to \CC^{E_0}$, the fiber $\pi^{-1}_{E_0}(\omega)$ is finite. Suppose, for a contradiction, that this is not the case; then the fiber contains another generic $d$-stress $\omega' \in M^*_{d,G}$. Let us fix generic frameworks $(G,p)$ and $(G,q)$ in $\CC^d$ with $\omega \in S(G,p)$ and $\omega' \in S(G,q)$.
    
    Note that $\omega$ and $\omega'$ can only differ on edges in $E - E_0$. This means that the framework $(G,q)$ satisfies the equilibrium conditions \cref{eq:equilibrium} associated to $\omega$ on all but at most $d+1$ vertices. Using \cref{lemma:stressalmosteveryvertex}, we deduce that $\omega \in S(G,q)$, and hence $\omega - \omega'$ is a nonzero stress of $(G,q)$ supported on $E - E_0$. This is a contradiction, since any generic framework on at most $d+1$ vertices is stress-free. 
\end{proof}

The second main result of this subsection is about the independent sets of the generic stress matroid.

\begin{theorem}\label{theorem:projectiveimages} \emph{(Small forests are independent in the stress matroid.)}\newline
    Let $G = (V,E)$ be a graph, and let $E_0 \subseteq E$ the edge set of a forest $F$ in $G$ on at most $d+1+c$ vertices, where $c$ denotes the number of components of $F$.
    If none of the edges in $E_0$ are $\mathcal{R}_d$-bridges in $G$, then $E_0$ is independent in $\stressmatroid{d}{G}$.
\end{theorem}

Before giving the proof of \cref{theorem:projectiveimages}, we recall the definition of a projective transformation and introduce some related terminology. Let $(G,p)$ be a realization of the graph $G$ in $\CC^d$, and let $\tilde{p} \in (\CC^{d+1})^V$ denote the configuration obtained by augmenting each vector $p(v), v \in V(G)$ with the number $1$ in the last coordinate. Let us say that a nonsingular matrix $A \in \CC^{(d+1) \times (d+1)}$ is \emph{feasible} (for $(G,p)$) if $\lambda_v = {(A\tilde{p}(v))}_{d+1} \neq 0$ for every $v \in V(G)$. Given a feasible matrix $A$, let us consider the configuration $q \in (\CC^d)^V$ obtained by dividing each coordinate of $A\tilde{p}(v)$ by $\lambda_v$ and then deleting the last coordinate, for each $v \in V(G)$. We call the configuration $q$ obtained in this way the \emph{projective image of $p$ under $A$}. The key fact, first observed in~\cite{roth.whiteley_1981}, is that if $\omega \in \CC^E$ is a stress of $(G,p)$, then the vector $\omega' \in \CC^E$ defined by $\omega'_{uv} = \lambda_u \lambda_v \omega_{uv}$ 
is a stress of $(G,q)$. For convenience, let us say that the pair $(q,\omega')$ is the \emph{projective image of the pair $(p,\omega)$ under $A$}. Let $P(G,p,\omega)$ denote the set of projective images of $(p,\omega)$. We shall need the following technical lemma, whose proof we defer to \cref{appendix:prelims}.

\begin{lemma}\label{lemma:projectivetechnical}
    Let $G = (V,E)$ be a graph, let $(G,p)$ be a generic realization in $\CC^d$, and let $\omega$ be a generic $d$-stress with $\omega \in S(G,p)$. Consider the set $P(G,p,\omega) \subseteq (\CC^{d})^V \times \CC^E $ defined as above. %
    We have $\pi_2(P(G,p,\omega)) \subseteq M_{d,G}^*$, where $\pi_2 : (\CC^{d})^V \times \CC^E \rightarrow \CC^{E}$ denotes the projection onto the second factor.
\end{lemma}

\begin{proof}[Proof of \cref{theorem:projectiveimages}]
    Let $\omega$ be a generic $d$-stress of $G$, and let $(G,p)$ be a generic realization in $\CC^d$ with $\omega \in S(G,p)$. Note that $\omega_{uv} \neq 0$ for every $uv \in E_0$. Indeed, since $\omega$ is $M_{d,G}^*$-generic, if it satisfied the equation $x_{uv} = 0$, then every element of $M_{d,G}^*$ would satisfy it. But since $uv$ is not an $\mathcal{R}_d$-bridge of $G$, we know that there exist \quasigen\ $d$-stresses of $G$ that are nonzero on $uv$. 
    
    Let $P(G,p,\omega)$ be defined as above. By \cref{lemma:projectivetechnical}, the projection of $P(G,p,\omega)$ onto $\CC^E$ is contained in $M_{d,G}^*$. Hence, to show that $E_0$ is independent in $\stressmatroid{d}{G} = \algmat{M_{d,G}^*}$, it is enough to show that the projection of $P(G,p,\omega)$ onto $\CC^{E_0}$ is full-dimensional. In other words, it suffices to show that for a nonempty Zariski open set of vectors $\sigma = {(\sigma_e)}_{e \in E_0} \in \CC^{E_0}$, there is a projective image $(q,\omega')$ of $(p,\omega)$ such that $\omega'_e = \sigma_e$ for all $e \in E_0$. For now, let us fix a vector $\sigma \in \CC^{E_0}$ that is everywhere nonzero.
    
    Let us assume for convenience that $c=1$ (i.e., $E_0$ induces a tree) and $|V(E_0)| = d+2$. The proof of the general case is very similar. Let $u_0$ be a leaf vertex in $E_0$, let $v_0$ be its unique neighbor in $E_0$, and set $V_0 = |V(E_0)| - u_0$. Note that $|V_0| = d+1$. 
    
    Recall the notation introduced in the paragraph before \cref{lemma:projectivetechnical}. By genericity, the points $\{p(v), v \in V_0\}$ are affinely independent, and hence the ``lifted'' points $\{\tilde{p}(v), v\in V_0\}$ are linearly independent. This means that we can freely map these points to any collection of $d+1$ points in $\CC^{d+1}$ by using a linear transformation, and in particular, we can freely choose the values $\lambda_v, v \in V_0$ (provided that they are nonzero). Moreover, we can write $\tilde{p}(u_0) = \sum_{v \in V_0}\alpha_v \tilde{p}(v)$ for some nonzero scalars $\alpha_v \in \CC, v \in V_0$. (The fact that $\alpha_v \neq 0$ follows from genericity again.) 
    
    Our goal is to solve the system of equations
    \begin{equation}\label{eq:projective}
        \lambda_u \lambda_v \omega_{uv} = \sigma_{uv}, \hspace{2em} \forall uv \in E_0,
    \end{equation}
    where, by the preceding discussion, $\lambda_v, v \in V_0$ are variables whose value we can freely choose, while by linearity we have $\lambda_{u_0} = \sum_{v \in V_0} \alpha_v \lambda_v$.

    Let us write $\lambda_{v_0} = c$, where $c \in \CC$ is a nonzero scalar whose value we will fix later. For each vertex $v \in V_0 - v_0$, there is a unique path $v_0, v_1, \ldots, v_k = v$ in $F$. We now define $\lambda_v$ by
    \[\lambda_v = 
    \begin{cases}
    \frac{\sigma_{v_0v_1}}{\omega_{v_0v_1}} \cdot \frac{\omega_{v_1v_2}}{\sigma_{v_1v_2}} \cdot \frac{\sigma_{v_2v_3}}{\omega_{v_2v_3}} \cdots \frac{\sigma_{v_{k-1}v_k}}{\omega_{v_{k-1}v_k}} \cdot c  &\text{if } k \text{ is odd,}  \\ \frac{\omega_{v_0v_1}}{\sigma_{v_0v_1}} \cdot \frac{\sigma_{v_1v_2}}{\omega_{v_1v_2}} \cdot \frac{\omega_{v_2v_3}}{\sigma_{v_2v_3}} \cdots \frac{\sigma_{v_{k-1}v_k}}{\omega_{v_{k-1}v_k}} \cdot c^{-1} &\text{if } k \text{ is even.}
    \end{cases}\]

    It is easy to verify that with this definition, \cref{eq:projective} is satisfied for every $uv \in E_0 - u_0v_0$. Moreover, we have \[\lambda_{u_0} = \sum_{v \in V_0}\alpha_v \lambda_v = A \cdot c + B \cdot c^{-1}\] for some scalars $A,B \in \CC$, and hence \[\lambda_{u_0}\lambda_{v_0}\omega_{uv} = (Ac^2 + B) \omega_{uv}.\] If $A \neq 0$, then the right-hand side is surjective as a function of $c$, and hence with a suitable choice of $c$ it can be made to equal $\sigma_{u_0v_0}$. 
    
    Thus it only remains to argue that $A \neq 0$ for a nonempty Zariski open set of vectors $\sigma \in \CC^{E_0}$. After clearing denominators, the condition $A = 0$ leads to a polynomial equation $\Phi(\omega,\sigma) = 0$ involving $\omega$ and $\sigma$. Thus $A \neq 0$ holds on a Zariski open set, and we only need to argue that this set is nonempty, or equivalently, that the polynomial $\Phi$ is not identically zero. This can be seen by looking at a vertex $v \in V_0 - v_0$ whose distance from $v_0$ is odd and maximal among such vertices. Let $e$ be the first edge on the path from $v$ to $v_0$ in $F$. If we look at the expression of $\Phi(\omega,\sigma)$ obtained by clearing denominators in $A$, $\sigma_e$ only appears in a single term. Hence the corresponding variable has positive degree in $\Phi$, and thus $\Phi \neq 0$.
\end{proof}

\cref{theorem:projectiveimages} is best possible in the sense that the edge set of a tree on $d+3$ vertices need not be independent in the stress matroid. For example, recall from \cref{example:K55} that the vertex stars are dependent in $\mathcal{S}_3(K_{5,5})$.  On the other hand, the proof technique of \cref{theorem:projectiveimages} seems quite weak in that it only uses projective images of $(p,\omega)$, and intuitively these form a small-dimensional subset of all pairs $(q,\omega')$ with $\omega' \in S(G,q)$. The problem is that even though we expect there to be many \quasigen\ stresses $\omega'$, it is not clear how to control the values of $\omega'$ on a given set of edges. 

In fact, we will only need the following corollary of \cref{theorem:projectiveimages}.

\begin{corollary}\label{theorem:stressindependence}
    For any graph $G$, the stress matroid $\stressmatroid{d}{G}$ does not contain circuits of size two.
\end{corollary}
\begin{proof}
    Consider a pair $\{e_1,e_2\} \subseteq E(G)$. 
    If $e_i$ is an $\mathcal{R}_d$-bridge in $G$ for some $i \in \{1,2\}$, then by \cref{corollary:degeneratedual}, $e_i$ is a loop in $\stressmatroid{d}{G}$, and hence $\{e_1,e_2\}$ is not a circuit. Otherwise, \cref{theorem:projectiveimages} implies that $\{e_1,e_2\}$ is independent in $\stressmatroid{d}{G}$ and hence, again, it is not a circuit.
\end{proof}

\subsection{The stress matroid of globally rigid graphs}\label{subsection:globallyrigidstressmatroid}

To close this section, we focus on the case of globally $d$-rigid graphs. For such a graph $G = (V,E)$, we can use \cref{theorem:ght} to obtain a concrete description of the contact locus $\contactlocus{d}{G}(\omega)$ for a generic $d$-stress $\omega \in \CC^E$, and use this description and \cref{theorem:contactlocusmatroid} to study the generic stress matroid $\stressmatroid{d}{G}$.

Fix a generic framework $(G,p)$ in $\CC^d$, and let $\omega \in \CC^E$ be a generic $d$-stress with $\omega \in S(G,p)$. \cref{theorem:ght}(b) implies that $\stresskernel{d}{G}(\omega) = \affine{G}{p}$; indeed, this is equivalent to the algebraic condition that the stress matrix associated to $\omega$ has rank $|V|-d+1$, so if it holds for some stress $\omega' \in S(G,p)$, then it holds for every generic $d$-stress. Now we can use \cref{lemma:contactlocusstresskernel} to deduce that
\begin{equation}\label{eq:globallyrigidcontactlocus}
    \contactlocus{d}{G}(\omega) = \overline{m_{d,G}(\stresskernel{d}{G}(\omega))} = \overline{m_{d,G}(\affine{G}{p})}.
\end{equation}

This suggests investigating the space $\affineedge{G}{p} = m_{d,G}(\affine{G}{p})$  of edge measurements of affine images of $(G,p)$. It turns out that $\affineedge{G}{p}$ is also a linear space, which we can describe concretely. Let us define the edge vectors $e_{uv} = p(u) - p(v)$ for each $uv \in E$. (Note that $e_{uv}$ is only defined up to sign; to be more precise, we could take an arbitrary orientation of $G$ and define $e_{uv}$ using this orientation.) Let $\sym{d}$ denote the space of symmetric $d \times d$ matrices over $\CC$.
\begin{lemma}\label{lemma:affineedgedim}\cite[Lemma 4.6]{gortler.etal_2019} \emph{(Linearity of affine edge measurements.)}\newline
    Let $(G,p)$ and $\{e_{uv}, uv \in E\}$ be as in the preceding discussion. Then
    \[\affineedge{G}{p} = \{{(e_{uv}^TQe_{uv})}_{uv \in E} : Q \in \sym{d}\} \subseteq \CC^E.\]In particular, $\affineedge{G}{p}$ is a linear space of dimension at most $\dim(\sym{d}) = \binom{d+1}{2}$.
\end{lemma}
We say that \emph{the edge directions of $(G,p)$ lie on a conic at infinity} if there is a nonzero symmetric matrix $Q \in \sym{d}$ such that $e_{uv}^TQe_{uv} = 0$ for every $uv \in E$. This is the same as saying that the linear mapping $Q \mapsto {(e_{uv}^T Q e_{uv})}_{uv \in E}$ is singular, which by \cref{lemma:affineedgedim} is equivalent to the condition that $\dim(\affineedge{G}{p}) < \binom{d+1}{2}$. The following is a classic result of Connelly. %
\begin{proposition}\label{theorem:conicatinfinity} (\cite[Proposition 4.3]{connelly_2005})
    Let $G$ be a graph and let $(G,p)$ be a generic realization in $\CC^d$. If $G$ has a subgraph $G_0$ in which every vertex has degree at least $d$, then the edge directions of $(G,p)$ do not lie on a conic at infinity.
\end{proposition}

\begin{proposition}\label{lemma:affinematroid}
    Let $G$ be a graph. The matroid $\algmat{\affineedge{G}{p}}$ is the same for every generic realization $(G,p)$ in $\CC^d$. Its rank is at most $\binom{d+1}{2}$, and is equal to $\binom{d+1}{2}$ if $G$ is not $\mathcal{R}_d$-independent.
\end{proposition}
\begin{proof}
    Recall that by \cref{lemma:linearalgebraicmatroid}, the algebraic matroid of a linear space is the same as the row matroid of any matrix whose columns span the space. By fixing a basis of $\sym{d}$ consisting of integer-valued matrices and considering the linear transformation $\sym{d} \rightarrow \affineedge{G}{p}$ described in \cref{lemma:affineedgedim} in this basis, we obtain a matrix $M(G,p)$ whose column space is $\affineedge{G}{p}$. Since the minors of $M(G,p)$ are polynomials with rational coefficients in the coordinates of $\{p(v), v \in V(G)\}$, the rank of each submatrix of $M(G,p)$ is the same for every generic choice of $p$, and thus so is the row matroid of $M(G,p)$.   
    
    Recall that the rank of $\algmat{\affineedge{G}{p}}$ is just $\dim(\affineedge{G}{p})$, which is at most $\binom{d+1}{2}$ by \cref{lemma:affineedgedim}.  If $G$ is not $\mathcal{R}_d$-independent, then it contains an $\mathcal{R}_d$-circuit $C$ as a subgraph. It is well-known (and can be deduced from \cref{lemma:whiteleygluing}(c)) that such a graph has minimum degree at least $d+1$. Hence by \cref{theorem:conicatinfinity}, the edge directions of $(G,p)$ do not lie on a conic at infinity for any generic realization $(G,p)$. This means by \cref{lemma:affineedgedim} that $\affineedge{G}{p}$ is linearly isomorphic to $\sym{d}$, which has dimension $\binom{d+1}{2}$.
\end{proof}

Following \cref{lemma:affinematroid}, we define the \emph{($d$-dimensional generic) affine edge measurement matroid} of $G$, denoted by $\affinematroid{d}{G}$, as \[\affinematroid{d}{G} = \algmat{\affineedge{G}{p}} \hspace{2em} \text{for any generic } (G,p) \text{ in } \CC^d.\] 
Note that a set of edges $E_0 \subseteq E$ is spanning in $\affinematroid{d}{G}$ if and only if for any generic framework $(G,p)$ in $\CC^d$, the edge directions of the subframework $(G[E_0],p_0)$ corresponding to $E_0$ do not line a conic at infinity.

We can now return to the discussion at the start of the subsection and relate $\affinematroid{d}{G}$ to the generic stress matroid $\stressmatroid{d}{G}$.

\begin{theorem}\label{theorem:affinematroidstressmatroid}
    Let $G$ be a graph. We have $\affinematroid{d}{G} \preceq \stressmatroid{d}{G}^*$. Moreover, if $G$ is \rig{d}, then $\affinematroid{d}{G} = \stressmatroid{d}{G}^*$ holds if and only if $G$ is \globrig{d}. 
\end{theorem}
\begin{proof}
    Let $\omega$ be a generic $d$-stress of $G$, and let $(G,p)$ be a generic framework in $\CC^d$ with $\omega \in S(G,p)$. Since $\affine{G}{p} \subseteq \stresskernel{d}{G}(\omega)$, we have (using \cref{lemma:contactlocusstresskernel}) \[\affineedge{G}{p} = m_{d,G}(\affine{G}{p}) \subseteq \overline{m_{d,G}(\stresskernel{d}{G}(\omega))} = \contactlocus{d}{G}(\omega).\]Combined with \cref{theorem:rigiditycontactlocusmatroid}, this implies \[\affinematroid{d}{G} = \algmat{\affineedge{G}{p}} \preceq \algmat{\contactlocus{d}{G}(\omega)} = \stressmatroid{d}{G}^*,\]
    as claimed.
    
    Now suppose that $G$ is \rig{d}. If $G$ is \globrig{d}, then we have $\affine{G}{p} = \stresskernel{d}{G}(\omega)$ by \cref{theorem:ght}. Hence $\affineedge{G}{p} = \contactlocus{d}{G}(\omega)$, so $\affinematroid{d}{G} = \stressmatroid{d}{G}^*$. Suppose now that $G$ is not \globrig{d}. By \cref{theorem:ght}, we have $\affine{G}{p} \subsetneq \stresskernel{d}{G}(\omega)$, and in particular $\dim(\stresskernel{d}{G}(\omega)) > d(d+1)$. (Note that since $G$ is \rig{d} but not \globrig{d}, it necessarily has at least $d+2$ vertices.) We shall show that $\dim(\contactlocus{d}{G}(\omega)) > d(d+1) - \binom{d+1}{2} = \binom{d+1}{2}$, and hence the rank of $\stressmatroid{d}{G}^*$ is larger than the rank of $\affinematroid{d}{G}$.
    
    To this end, let $n = |V(G)|$ and consider the set $R_p = m_{d,G}^{-1}(m_{d,G}(p))$ of configurations $q$ such that $(G,q)$ is equivalent to $(G,p)$. By \cref{lemma:genericproperties}(b) and the assumption that $G$ is \rig{d}, we have \[\dim(R_p) = dn - \dim(M_{d,G}) = dn - dn + \binom{d+1}{2} = \binom{d+1}{2}.\]
    Now we make the same computation for the restriction $m_{d,G} : \stresskernel{d}{G}(\omega) \rightarrow \contactlocus{d}{G}(\omega)$. The key observation is that by \cref{theorem:connelly}, we have $R_p \subseteq \stresskernel{d}{G}(\omega)$, and hence the fiber of this restricted map over $m_{d,G}(p)$ is still $R_p$. It follows that 
    \begin{align*}
        \dim(\contactlocus{d}{G}(\omega)) = \dim(\stresskernel{d}{G}(\omega)) - \dim(R_p) > d(d+1) - \binom{d+1}{2} = \binom{d+1}{2} = \dim(\affineedge{G}{p}),
    \end{align*}
    as desired.
    
\end{proof}

In particular, \cref{lemma:affinematroid,theorem:affinematroidstressmatroid} imply that if $G$ is not $\mathcal{R}_d$-independent, then the rank of $\stressmatroid{d}{G}^*$ is at least $\binom{d+1}{2}$, and hence the rank of $\stressmatroid{d}{G}$ is at most $|E| - \binom{d+1}{2}$. 

In the previous subsection we proved that sufficiently large edge sets are always spanning in $\stressmatroid{d}{G}$ (\cref{corollary:stressmatroidspanning}); dually, ``small'' edge sets are independent in $\stressmatroid{d}{G}^*$, where small means that the edge set spans at most $d+1$ vertices. We now show that the same holds for the affine edge measurement matroid $\affinematroid{d}{G}$. Combined with the fact that $\affinematroid{d}{G} \preceq \stressmatroid{d}{G}^*$, this gives an alternative proof of \cref{corollary:stressmatroidspanning}.

\begin{proposition}\label{theorem:affineindependent}
    Let $G$ be a graph and let $E_0 \subseteq E(G)$. If $E_0$ spans at most $d+1$ vertices, then it is independent in $\affinematroid{d}{G}$.
\end{proposition}
\begin{proof}
    Let $G_0 = G[E_0]$. Consider a generic realization $(G,p)$ of $G$ in $\CC^d$, and let $(G_0,p_0)$ be the subframework corresponding to $G_0$. Since $|V(G_0)| \leq d+1$, every configuration $q_0 \in (\CC^d)^{V(G_0)}$ arises as an affine image of $p_0$, and hence $\affine{G_0}{p_0} = (\CC^d)^{V(G_0)}$. Every graph on at most $d+1$ vertices are $\mathcal{R}_d$-independent, so we have \[\overline{m_{d,G_0}(\affine{G_0}{p_0})} = \overline{m_{d,G_0}({(\CC^{d})}^{V(G_0)})} = M_{d,G_0} = \CC^{E_0}.\] But clearly \[\pi_{E_0}(\affineedge{G}{p}) = \pi_{E_0}(m_{d,G}(\affine{G}{p})) = m_{d,G_0}(\affine{G_0}{p_0}),\] so we have $\overline{\pi_{E_0}(\affineedge{G}{p})} = \CC^{E_0}$, which shows that $E_0$ is independent in $\algmat{\affineedge{G}{p}} = \affinematroid{d}{G}$.
\end{proof}

When $d=1$, \cref{lemma:affinematroid,theorem:affineindependent} give a complete description of $\affinematroid{d}{G}$.

\begin{corollary}\label{lemma:1dimaffinematroid}
    For any graph $G$, $\affinematroid{1}{G}$ is the rank $1$ uniform matroid on $E(G)$. 
\end{corollary}

Using the preceding results we can determine the one-dimensional generic stress matroid of any graph $G$. We will need the specialization of \cref{lemma:dualconnectivity} to the measurement variety.
\begin{proposition}\label{lemma:stressmatroidconnectivity}
    Let $G$ be a graph, and let $G_i, i \in \{1,\ldots,k\}$ be the $\mathcal{R}_d$-components of $G$. We have $M_{d,G}^* = M_{d,G_1}^* \times \cdots \times M_{d,G_k}^*$, and hence $\stressmatroid{d}{G} = \bigoplus_{i=1}^{k} \stressmatroid{d}{G_i}$.
\end{proposition}

\begin{proposition}\label{theorem:1dimstressmatroid}\emph{(Characterization of the one-dimensional stress matroid.)}\newline
    Let $G = (V,E)$ be a graph, and let $E_1,\ldots,E_k$ be the edge sets of its $2$-connected components. Let $m_i = |E_i|$ for $i \in \{1,\ldots,k\}$. Then $\stressmatroid{1}{G} = \bigoplus_{i=1}^k U_{m_i-1}(E_i)$, where $U_{k}(X)$ denotes the rank $k$ uniform matroid on ground set $X$. 
\end{proposition}
\begin{proof}
    The $\mathcal{R}_1$-components of a graph are just its $2$-connected components. Using \cref{lemma:stressmatroidconnectivity}, we may assume that $G$ is $2$-connected (i.e., \globrig{1}). From \cref{theorem:affinematroidstressmatroid}, we have that $\stressmatroid{1}{G}$ is the dual of $\affinematroid{1}{G}$, and by \cref{lemma:1dimaffinematroid}, $\affinematroid{1}{G} = U_1(E)$, so $\stressmatroid{1}{G} = U_{|E|-1}(E)$, as required. 
\end{proof}

To close this section, let us consider the matroid $\affinematroid{d}{G}$ for $d \geq 2$.
A computer-assisted proof shows that $\affinematroid{2}{G}$ is the uniform matroid of rank $3 = \binom{2+1}{2}$ for any graph with at least $3$ edges. However, in higher dimensions it is not true that $\affinematroid{d}{G}$ is uniform of rank $\binom{d+1}{2}$. To see this, observe that if $G$ is a graph consisting of two copies of $K_d$, then for any framework $(G,p)$, the edge directions of $(G,p)$ lie on a conic at infinity. (Intuitively, the edge directions all belong to the union of two hyperplanes, which is a degenerate conic.) Hence for any graph $G$, any edge set that can be covered by two copies of $K_d$ has rank less than $\binom{d+1}{2}$ in $\affinematroid{d}{G}$. We go out on a limb and conjecture that this condition characterizes the dependent sets in $\affinematroid{d}{G}$, given that $G$ has at least $\binom{d+1}{2}$ edges.\footnote{The previous version of this paper contained the stronger conjecture that $\affinematroid{d}{G}$ is uniform whenever $G$ has at least $\binom{d+1}{2}$ edges. I am grateful to Steven Gortler for pointing out that this is false.} 

\begin{conjecture}\label{conj:affinematroid}
    Given a graph $G$, a set of edges $E_0 \subseteq E(G)$ with $|E_0| \geq \binom{d+1}{2}$ is spanning in $\affinematroid{d}{G}$ if and only if $G[E_0]$ cannot be written as the subgraph of the union of two (not necessarily disjoint) copies of $K_d$.
\end{conjecture}

We can verify that for $d \leq 3$, the condition in \cref{conj:affinematroid} indeed defines a matroid; however, for $d \geq 4$, even this is unclear.
Note that \cref{conj:affinematroid} would strengthen \cref{theorem:conicatinfinity} by giving a complete characterization of graphs $G$ for which the edge directions of a generic realization $(G,p)$ does not lie on a conic at infinity.

\section{Stress-linked vertex pairs}\label{section:globallylinked}

We now begin investigating the central new idea of this paper: stress-linked vertex pairs. %
Let us note that the results of this section are largely independent of the previous section, except for \cref{theorem:nongloballylinked,theorem:bridgegloballylinked,theorem:gluing} below, whose proofs use the structural results of \cref{subsection:stressmatroid,subsection:spanningindependent} about the measurement variety and the generic stress matroid.

Let $G$ be a graph. Recall that $\stressnullity{G}$ denotes the $d$-dimensional shared stress nullity of $G$, which is the dimension of the one-dimensional shared stress kernel $K_1(G,p)$ for some generic framework $(G,p)$ in $\CC^d$. 
Now let $u,v \in V(G)$ be a pair of vertices. 
We say that $\{u,v\}$ is \emph{\strongly{d} in $G$} if $\{u,v\}$ is $\mathcal{R}_d$-linked in $G$ and $\stressnullity{G+uv} = \stressnullity{G}$. Note that if $uv \in E(G)$, then $\{u,v\}$ trivially satisfies these conditions. 

Intuitively, for an $\mathcal{R}_d$-linked pair of vertices $\{u,v\}$, being $d$-stress-linked means that for every generic framework $(G,p)$, the stress space $S(G+uv,p)$ is uniquely determined already by $S(G,p)$, without needing to know $(G,p)$ itself. This intuition is formalized by part \emph{(b)} of the following proposition, which rephrases the definition of being \strongly{d} in various convenient forms. 
\begin{proposition}\label{lemma:def} \emph{(Equivalent definitions of stress-linked pairs.)}\newline
    Let $G$ be a graph and let $u,v \in V(G)$ be a pair of nonadjacent vertices such that $\{u,v\}$ is $\mathcal{R}_d$-linked in $G$. Let $(G,p)$ be a generic framework in $\CC^d$, and let $\omega \in S(G+uv,p)$ be a stress of $(G+uv,p)$ that is nonzero on $uv$. The following are equivalent.
    \begin{enumerate}
        \item $\{u,v\}$ is \strongly{d} in $G$,
        \item $\sharedstresskernel{d}(G+uv,p) = \sharedstresskernel{d}(G,p)$,
        \item $\omega$ is a stress of $(G+uv,q)$, for every configuration $q \in \sharedstresskernel{d}(G,p)$,
        \item $\omega$ is a stress of $(G+uv,q)$, for every generic configuration $q \in \sharedstresskernel{d}(G,p)$.
    \end{enumerate}
\end{proposition}
\begin{proof}
    Since the pair $\{u,v\}$ is $\mathcal{R}_d$-linked, it is \strongly{d} if and only if $\stressnullity{G+uv} = \stressnullity{G}$. This is equivalent to $\dim(K_1(G+uv,p)) = \dim(K_1(G,p))$. Since $K_1(G+uv,p) \subseteq K_1(G,p)$ holds by definition, equality of dimensions is equivalent to $K_1(G+uv,p) = K_1(G,p)$. Using the observation that $\sharedstresskernel{d}(G,p) = {(K_1(G,p))}^d$, this is further equivalent to $\sharedstresskernel{d}(G+uv,p) = \sharedstresskernel{d}(G,p)$. This shows the equivalence of  \textit{(a)} and \textit{(b)}. 
    
    Using a slight abuse of notation, let us identify $S(G,p)$ with the subspace of $S(G+uv,p)$ consisting of stresses that vanish on $uv$. A dimension count shows that $S(G+uv,p)$ is generated by $S(G,p) \cup \{\omega\}$, and hence  $\sharedstresskernel{d}(G+uv,p) = \sharedstresskernel{d}(G,p) \cap \stresskernel{d}{G+uv}(\omega)$. In other words,  $\sharedstresskernel{d}(G+uv,p) = \sharedstresskernel{d}(G,p)$ holds if and only if $\sharedstresskernel{d}(G,p) \subseteq \stresskernel{d}{G+uv}(\omega)$, which is the content of \textit{(c)}. This shows the equivalence of \textit{(b)} and \textit{(c)}. 
    
    Condition \textit{(c)} clearly implies \textit{(d)}. For the reverse direction, note that since $(G,p)$ is generic, the linear space $\sharedstresskernel{d}(G,p)$ contains a generic configuration, and hence by \cref{lemma:genericdense}, the generic configurations are Zariski dense in it. The condition that $\omega$ is a stress of $(G+uv,q)$ is described by polynomial equations, so if it holds for a Zariski dense subset of $\sharedstresskernel{d}(G,p)$, then it holds for all of it.
\end{proof}

Our main motivation for studying \strongly{d} vertex pairs is the following theorem, which is one of the main results of this paper.

\begin{theorem}\label{theorem:nongloballylinked}
    Let $G$ be a graph and let $u,v \in V(G)$ be a pair of vertices. If $\{u,v\}$ is \strongly{d} in $G$, then $\{u,v\}$ is globally linked in $G$ in $\CC^d$.
\end{theorem}
\begin{proof}
    Suppose, in the contrapositive, that $\{u,v\}$ is $\mathcal{R}_d$-linked in $G$ but not globally linked in $G$ in $\CC^d$. In particular, $u$ and $v$ must necessarily be nonadjacent in $G$. Let $(G,p)$ and $(G,q)$ be equivalent frameworks in $\CC^d$ with $(G,p)$ generic and such that $(G+uv,p)$ and $(G+uv,q)$ are not equivalent. Let $G' = G+uv$, and let $\omega'$ be a generic $d$-stress of $G'$ with $\omega' \in S(G',p)$. By \cref{theorem:connelly}, $q \in \sharedstresskernel{d}(G,p)$. Thus by \cref{lemma:def}(c), in order to show that $\{u,v\}$ is not \strongly{d} in $G$, it is enough to show that $\omega'$ is not a stress of $(G',q)$. 
    
    Suppose, for a contradiction, that it is. Then $q \in \stresskernel{d}{G'}(\omega')$, and hence the contact locus $\contactlocus{d}{G'}(\omega') = \overline{m_{d,G'}(\stresskernel{d}{G'}(\omega'))}$ contains both $m_{d,G'}(p)$ and $m_{d,G'}(q)$ (cf.\ \cref{lemma:contactlocusstresskernel}). Since $\contactlocus{d}{G'}(\omega') \subseteq M_{d,G'}$ is a linear space, this means that the fiber of the projection $\pi_E : M_{d,G'} \rightarrow M_{d,G}$ over $m_{d,G}(p)$  is all of $\{m_{d,G}(p)\} \times \CC$; in particular, it is one-dimensional. Since $m_{d,G'}(p)$ is $M_{d,G'}$-generic by \cref{lemma:genericimage}, \cref{lemma:genericproperties}(b) now implies that $\dim(M_{d,G'}) - \dim(M_{d,G}) = 1$, or equivalently, that $r_d(G') = r_d(G) + 1$. This means that $\{u,v\}$ is not $\mathcal{R}_d$-linked in $G$, a contradiction. 
\end{proof}

It is natural to also consider the converse of \cref{theorem:nongloballylinked}: is it true that if $\{u,v\}$ is globally linked in $G$ in $\CC^d$ (or in $\RR^d$), then $\{u,v\}$ is \strongly{d}? We shall return to this question in \cref{section:concluding}. We note that the answer is affirmative when every pair of vertices is globally linked, that is, when $G$ is \globrig{d}.

\begin{proposition}\label{lemma:globallyrigidstronglylinked}
    A graph $G$ is \globrig{d} if and only if $\{u,v\}$ is \strongly{d} in $G$ for all pairs $u,v \in V(G)$.
\end{proposition}
\begin{proof}
    The ``if'' direction follows from \cref{theorem:nongloballylinked}. For the other direction, let us assume that $G$ is \globrig{d}.
    If $G$ has at most $d+1$ vertices, then it is complete, and the statement is vacuous. Otherwise, it follows from \cref{theorem:ght} that for every pair of vertices $u,v \in V(G)$, we have $d+1 = \stressnullity{G} = \stressnullity{G+uv}$. Since $G$ is \globrig{d}, it is also \rig{d}, and hence every pair of vertices is $\mathcal{R}_d$-linked in $G$. This shows that every pair of vertices is \strongly{d} in $G$.
\end{proof}

We can also characterize \strongly{d} vertex pairs in terms of so-called Gauss fibers. Given a generic framework $(G,p)$ in $\CC^d$, its \emph{Gauss fiber} is the set $L(G,p) = \overline{m_{d,G}(\sharedstresskernel{d}(G,p))}$. It can be shown that the Gauss fiber of $(G,p)$ is the intersection of the contact loci $\contactlocus{d}{G}(\omega)$ for $\omega \in S(G,p)$. In particular, it is a linear space. The following statement can be shown by a dimension counting argument similar to the one we saw in the proof of \cref{theorem:affinematroidstressmatroid}. Since we shall not use this result, we defer its proof to \cref{appendix:stresslinked}.

\begin{proposition}\label{lemma:gaussfiberchar} \emph{(Geometric characterization of stress-linked pairs.)}\newline
    Let $G$ be a graph, let $(G,p)$ be a generic framework in $\CC^d$, and let $u,v \in V(G)$ be a pair of vertices. The vertex pair $\{u,v\}$ is \strongly{d} in $G$ if and only if $\dim(L(G+uv,p)) = \dim(L(G,p))$.
\end{proposition}

\subsection{Basic properties of stress-linked pairs}

In this subsection we establish a number of basic results about \strongly{d} vertex pairs. We start with the most interesting one: that $\mathcal{R}_d$-bridges and \strongly{d} pairs ``do not interact'' with each other. More precisely, we shall show that deleting an $\mathcal{R}_d$-bridge of a graph preserves its \strongly{d} pairs, and similarly, adding an edge between a \strongly{d} vertex pair preserves its $\mathcal{R}_d$-bridges. We need the following easy observation.

\begin{lemma}\label{lemma:bridgeconverse}
    Let $G$ be a graph and let $u,v \in V(G)$ be a nonadjacent pair of vertices. If $\{u,v\}$ is not $\mathcal{R}_d$-linked in $G$, then $\stressnullity{G} = \stressnullity{G+uv}$.
\end{lemma}
\begin{proof}
    Let $(G,p)$ be a generic realization in $\CC^d$. The assumption on $\{u,v\}$ means that $S(G+uv,p)$ only consists of stresses supported on $E$. It follows that $\sharedstresskernel{d}(G,p) = \sharedstresskernel{d}(G+uv,p)$, and thus $\stressnullity{G} = \stressnullity{G+uv}$.
\end{proof}

\begin{theorem}\label{theorem:bridgegloballylinked}\emph{(Stress-linked pairs and bridges.)}\newline
    Let $G = (V,E)$ be a graph, and let us fix an edge $e \in E$ and a nonadjacent pair of vertices $u,v \in V$ of $G$. Suppose that $e$ is an $\mathcal{R}_d$-bridge in $G$ and $\{u,v\}$ is \strongly{d} in $G$. Then $\{u,v\}$ is \strongly{d} in $G-e$.
\end{theorem}
\begin{proof}
    Let $G' = G + uv$. The assumptions on $e$ and $\{u,v\}$ combined with \cref{lemma:bridgeconverse} imply that $\stressnullity{G-e} = \stressnullity{G} = \stressnullity{G'}$, and since $\stressnullity{G-e} \geq \stressnullity{G'-e} \geq \stressnullity{G'}$, it follows that $\stressnullity{G-e} = \stressnullity{G'-e}$. Thus we only need to show that $\{u,v\}$ is $\mathcal{R}_d$-linked in $G-e$. For a contradiction, suppose that this is not the case. Then $r_d(G) = r_d(G'-e) = r_d(G')$, so neither $e$ nor $uv$ is an $\mathcal{R}_d$-bridge in $G'$. 
    
    Let $(G,p)$ be a generic framework of $G$ in $\CC^d$, and let $\omega'$ be a generic $d$-stress of $(G',p)$. We recall \cref{theorem:rigiditycontactlocusmatroid}, which states that the algebraic matroid $\algmat{\contactlocus{d}{G'}(\omega')}$ of the contact locus $\contactlocus{d}{G'}(\omega')$ is dual to the generic stress matroid $\stressmatroid{d}{G'}$. Note that since neither $e$ nor $uv$ is an $\mathcal{R}_d$-bridge in $G'$, by \cref{corollary:degeneratedual}, they are not loops in $\stressmatroid{d}{G'}$. We shall show that $E - e$ is not spanning in $\algmat{\contactlocus{d}{G'}(\omega')}$, and hence by duality $\{e,uv\}$, is a circuit in $\stressmatroid{d}{G'}$, contradicting \cref{theorem:stressindependence}. 
    
    Let $R_p = m_{d,G-e}^{-1}(m_{d,G-e}(p))$ denote the set of configurations $q$ for which $(G-e,q)$ is equivalent to $(G-e,p)$. \cref{theorem:connelly} and the fact that $(G-e,p)$ is generic imply that $R_p \subseteq \sharedstresskernel{d}(G-e,p)$. On the other hand, since $\stressnullity{G-e} = \stressnullity{G'}$, we have $\sharedstresskernel{d}(G-e,p) = \sharedstresskernel{d}(G',p)$, from which we deduce that
    \begin{equation}\label{eq:fiberstresskernel}
        R_p \subseteq \sharedstresskernel{d}(G-e,p) = \sharedstresskernel{d}(G',p) \subseteq \stresskernel{d}{G'}(\omega').    
    \end{equation}
     By \cref{lemma:contactlocusstresskernel}, we have $\contactlocus{d}{G'}(\omega') = \overline{m_{d,G'}(\stresskernel{d}{G'}(\omega'))}$, so \cref{eq:fiberstresskernel} implies $m_{d,G'}(R_p) \subseteq \contactlocus{d}{G'}(\omega)$. In particular, the fiber $F$ of the projection $\pi_{E-e} : \contactlocus{d}{G'}(\omega') \rightarrow M_{d,G-e}$ over $m_{d,G-e}(p)$ contains $m_{d,G'}(R_p)$. (See \cref{fig:bridgeproof}.) Since $e$ is an $\mathcal{R}_d$-bridge in $G$, $m_{d,G'}(R_p)$ is an infinite set, and hence so is the fiber $F$. Since $\contactlocus{d}{G'}(\omega')$ is a linear space, \cref{lemma:linearspanning} now implies that $E-e$ is not spanning in $\algmat{\contactlocus{d}{G'}(\omega')}$, and thus $\{e,uv\}$ is a circuit in $\stressmatroid{d}{G'}$, a contradiction.     
\end{proof}

\usetikzlibrary{decorations.pathmorphing}
\usetikzlibrary{calc}
\tikzset{snake it/.style={decorate, decoration=snake}}

\begin{figure}[t]
    \centering
        \begin{tikzpicture}[x = 1cm, y = 1cm, scale = 1.3]
            \node (1) at (0,0){};
            \node (2) at ($(1) + (3,0)$){};
            \node (3) at ($(2) + (1.1,1.1)$){};
            \node [label={[label distance=0pt]90:$K_{d,G'}(\omega)$}](4) at ($(1) + (1.1,1.1)$){};
        \draw [line width=\normaledge,color=edgeblack] (1.center) -- (2.center) -- (3.center) -- (4.center) -- (1.center);

        \node (rp) at ($(1) + (.9,.22)$){$R_p$};

            \node [label={[label distance=-7pt]above left:$p$}](p) at (2,.65){};
            \draw [fill=vertexblack] (p) circle (1.5pt);
 
        \draw [black, line width=.7] plot [smooth] coordinates {(0.3,0.3)  (p) (3.3,0.3)};
        
            \node (5) at (7.7,1.3){};
            \node (6) at ($ (5) + (3,.8) $){};
            \node (7) at ($(6) + (0,1.4)$){};
            \node [label={[label distance=0pt]90:$L_{d,G'}(\omega)$}](8) at ($(5) + (0,1.4)$){};
        \draw [line width=\normaledge,color=edgeblack] (5.center) -- (6.center) -- (7.center) -- (8.center) -- (5.center);  

        \node (56h) at ($(5)!0.5!(6)$) {};
        \node (78h) at ($(7)!0.5!(8)$) {};

        \node [label={[label distance=-7pt]below left:$m_{d,G'}(p)$}](p2) at ($(56h)!0.5!(78h)$){};
        \draw [fill=vertexblack] (p2) circle (1.5pt);
        \path [draw=black,line width=.8,decorate, decoration={snake, amplitude=.25mm}]  (56h.center) -- (78h.center) node[midway,label={[label distance=2pt]85:$m_{d,G'}(R_p)$}]{};

        \draw [->] (4.6,1.3) -- (7,1.9) node[midway,sloped,above]{$m_{d,G'}$};

        \draw [->] (4.6,0.1) -- (7,-.4) node[midway,sloped,below]{$m_{d,G-e}$};

        \node (bottomleft) at ($(5) + (0,-2)$){};
        \node (bottomright) at ($(6) + (0,-2)$){};
        \node [label={[label distance=-5pt]below right:$m_{d,G-e}(R_p)$}](p3) at ($(bottomleft)!0.5!(bottomright)$){};

        \draw [line width=.8,color=edgeblack] (bottomleft.center) -- (bottomright.center);  

        \draw [fill = vertexblack] (p3) circle (1.5pt);
         
        \draw [->,shorten <=0.5cm, shorten >=0.5cm] (56h) -- (p3) node[midway,right]{$\pi_{E - e}$};
        
        \end{tikzpicture}
    \caption{An illustration of the various spaces appearing in the proof of \cref{theorem:bridgegloballylinked}.}\label{fig:bridgeproof}
\end{figure}

\begin{theorem}\label{corollary:bridges} \emph{(Cocircuit theorem for stress-linked pairs.)} \newline
    Let $G$ be a graph and let $u,v \in V(G)$ be a pair of vertices. If $\{u,v\}$ is \strongly{d} in $G$, then the set of $\mathcal{R}_d$-bridges is the same in $G$ and $G+uv$. 
\end{theorem}
\begin{proof}
    Let $e$ be an edge of $G$. Clearly, if $e$ is an $\mathcal{R}_d$-bridge in $G+uv$, then it is also an $\mathcal{R}_d$-bridge in $G$. Conversely, if $e$ is an $\mathcal{R}_d$-bridge in $G$, then by \cref{theorem:bridgegloballylinked}, $\{u,v\}$ is \strongly{d}, and in particular $\mathcal{R}_d$-linked, in $G-e$. Thus we have $r_d(G+uv) = r_d(G) > r_d(G-e) = r_d(G + uv - e)$, which shows that $e$ is a bridge in $G+uv$.
\end{proof}

In matroidal terms, \cref{corollary:bridges} says that if $\{u,v\}$ is \strongly{d} in $G$, then $uv$ is not contained in any size two cocircuit of $\mathcal{R}_d(G+uv)$. As we will see, this property of stress-linked pairs makes them a useful tool even in problems unrelated to global rigidity or globally linked pairs.

We note that \cref{corollary:bridges} implies Hendrickson's theorem~\cite{hendrickson_1992}, which says that \globrig{d} graphs on at least $d+2$ vertices are redundantly \rig{d}, or in other words, they are \rig{d} and do not contain $\mathcal{R}_d$-bridges. Indeed, if $G$ is \globrig{d}, then by \cref{lemma:globallyrigidstronglylinked}, every pair of vertices is \strongly{d} in $G$, so by \cref{corollary:bridges}, the set of $\mathcal{R}_d$-bridges of $G$ is the same as the set of $\mathcal{R}_d$-bridges of the complete graph on $V(G)$, which is easily seen to be bridgeless when $|V(G)| \geq d+2$. Although at first sight this appears to be a novel proof of Hendrickson's theorem, we implicitly used \cref{theorem:ght} (through \cref{lemma:globallyrigidstronglylinked}), which quickly implies Hendrickson's theorem by itself.

As a simple corollary of \cref{lemma:largenullity}(a), we show that in an $\mathcal{R}_d$-independent graph, the only \strongly{d} vertex pairs are the ones corresponding to edges of the graph. It is an open question whether the analogous statement holds for globally linked pairs in $\RR^d$ (or in $\CC^d$). It is known to be true in the special case when $d=2$ and the graph is minimally \rig{2}, see~\cite[Corollary 2]{jackson.etal_2014}.

\begin{proposition}
    If $G$ is $\mathcal{R}_d$-independent, then a pair $\{u,v\}$ of vertices is \strongly{d} in $G$ if and only if $uv \in E(G)$.
\end{proposition}
\begin{proof}
    If $uv \in E(G)$, then $\{u,v\}$ is clearly \strongly{d} in $G$. Otherwise, either $\{u,v\}$ is not $\mathcal{R}_d$-linked, or $G+uv$ is not $\mathcal{R}_d$-independent. In the latter case, we have $\stressnullity{G+uv} <  |V(G)| = \stressnullity{G}$ by \cref{lemma:largenullity}, so $\{u,v\}$ is not \strongly{d}.
\end{proof}

We close this subsection by noting that being \strongly{d} is preserved under taking supergraphs and under separation along a complete graph. The (somewhat technical but routine) proofs can be found in \cref{appendix:stresslinked}.

\begin{lemma}\label{lemma:stronglylinkedsubgraph} \emph{(Subgraph lemma for stress-linked pairs.)}\newline
    Let $G$ be a graph, let $G_0$ be a subgraph of $G$, and let $u,v \in V(G_0)$ be a nonadjacent pair of vertices. If $\{u,v\}$ is \strongly{d} in $G_0$, then $\{u,v\}$ is \strongly{d} in $G$.
\end{lemma}

\begin{lemma}\label{lemma:completeseparator}\emph{(Clique separators and stress-linked pairs.)}\newline
    Let $G = (V,E)$ be the union of the graphs $G_1 = (V_1,E_1)$ and $G_2 = (V_2,E_2)$, and suppose that $G_1 \cap G_2$ is either empty or a complete graph. Let $u,v \in V_1$ be nonadjacent a pair of vertices. Then $\{u,v\}$ is \strongly{d} in $G$ if and only if $\{u,v\}$ is \strongly{d} in $G_1$.
\end{lemma}

\subsection{Stress-linked pairs and small separators}

In this subsection we answer a conjecture of Jordán and the author affirmatively by proving the following theorem.

\begin{theorem}\label{conjecture:gluing}\cite[Conjecture 5.6]{garamvolgyi.jordan_2023a} \emph{(Gluing theorem for globally linked pairs.)} \newline
    Let $G$ be the union of the graphs $G_1 = (V_1,E_1), G_2 = (V_2,E_2)$ with $|V_1 \cap V_2| \leq d+1$, and let $u,v \in V_1 \cap V_2$ be a pair of vertices. If $\{u,v\}$ is $\mathcal{R}_d$-linked in both $G_1$ and $G_2$, then $\{u,v\}$ is globally linked in $G$ in $\RR^d$.
\end{theorem}

\noindent We shall verify \cref{conjecture:gluing} in the following strong form.

\begin{theorem}\label{theorem:gluing} \emph{(Gluing theorem for stress-linked pairs.)}\newline
    Let $G$ be the union of the graphs $G_1 = (V_1,E_1), G_2 = (V_2,E_2)$ with $|V_1 \cap V_2| \leq d+1$, and let $u,v \in V_1 \cap V_2$ be a pair of vertices. If $\{u,v\}$ is $\mathcal{R}_d$-linked in both $G_1$ and $G_2$, then $\{u,v\}$ is \strongly{d} in $G$.
\end{theorem}
\begin{proof}
    Let us assume that $u$ and $v$ are nonadjacent in $G_1$ and $G_2$, for otherwise the statement is trivial. Fix a generic realization $(G,p)$, and let $(G_i,p_i), i \in \{1,2\}$ denote the subframeworks of $(G,p)$ corresponding to $G_1$ and $G_2$. By a slight abuse of notation, let us identify $S(G,p), S(G_1+uv,p_2)$ and $S(G_2+uv,p_2)$ with the subspace of $S(G+uv,p)$ of stresses supported on $E(G)$, $E(G_1+uv)$ and $E(G_2+uv)$, respectively. Since $\{u,v\}$ is $\mathcal{R}_d$-linked in $G_i$, there exists a stress $\omega_i$ of $(G_i+uv,p_i)$ with ${(\omega_i)}_{uv} = 1$, for $i \in \{1,2\}$. 

    Now let us fix a configuration $q \in \sharedstresskernel{d}(G,p)$, and let $(G_i,q_i), i \in \{1,2\}$ denote the corresponding subframeworks. By \cref{lemma:def}(c), to prove that $\{u,v\}$ is \strongly{d} in $G$, it suffices to show that $\omega_1$ is a stress of $(G_1 + uv,q_1)$. To this end, let us consider the stress $\omega = \omega_1 - \omega_2 \in S(G,p) \subseteq S(G,q)$. Note that $\omega$ and $\omega_1$ agree on all of $E_1 - E_2$; in particular, for every vertex $z \in V_1 - V_2$, $\omega$ and $\omega_1$ agree on every edge incident to $z$. Since $\omega \in S(G,q)$, this shows that $\omega_1$ and $(G_1+uv,q_1)$ satisfy the equilibrium conditions \cref{eq:equilibrium} at every vertex in $V_1 - V_2$. Since $|V_1 \cap V_2| \leq d+1$ and $\omega_1$ is a stress of the generic framework $(G_1+uv,p_1)$, \cref{lemma:stressalmosteveryvertex} now implies that $\omega$ is a stress of $(G_1+uv,q_1)$, as desired.
\end{proof}

\cref{conjecture:gluing} follows immediately from \cref{theorem:gluing} and \cref{theorem:nongloballylinked}.
Let us note that the bound $d+1$ is best possible in both theorems. Indeed, if $d=1$ and $G=G_1=G_2$ is a path on three vertices with endpoints $u,v$, then $\{u,v\}$ is $\mathcal{R}_1$-linked in both $G_1$ and $G_2$, but not \strongly{1} (nor globally linked) in $G$.

We may also ask the converse question: supposing that for some vertices $u,v\in V(G_1)$ the pair $\{u,v\}$ is \strongly{d} in $G_1 \cup G_2$, when can we deduce that it is also \strongly{d} in $G_1$? \cref{lemma:completeseparator} shows that this is the case when $G_1 \cap G_2$ is a complete graph. Below we consider the case when $V_1 \cap V_2$ has size two, and show that if the vertex pair in the intersection is not $\mathcal{R}_d$-linked in at least one of $G_1$ and $G_2$, then $\{u,v\}$ is also \strongly{d} in $G_1$.

It is interesting to consider the analogous question for globally linked pairs. If $V_1 \cap V_2 = \{x,y\}$ and $\{x,y\}$ is not $\mathcal{R}_d$-linked in $G_2$, then it can be shown that $\{u,v\}$ is globally linked in $G_1$ in $\CC^d$ if and only if it is globally linked in $G_1 \cup G_2$ in $\CC^d$. If we replace $\CC^d$ by $\RR^d$, then it is unclear whether the statement remains true; a special case of this question was posed as a conjecture in~\cite[Conjecture 5.7]{garamvolgyi.jordan_2023a}. The case when $\{x,y\}$ is not $\mathcal{R}_d$-linked in $G_1$ seems unclear even in the complex case.

\begin{theorem}\label{theorem:2separatorconverse} \emph{(Stress-linked pairs and 2-separators.)} \newline
    Let $G$ the union of the induced subgraphs $G_1 = (V_1,E_1)$ and $G_2 = (V_2,E_2)$ with $V_1 \cap V_2 = \{x,y\}$, and let $u,v \in V_1$ be a pair of vertices. Suppose that at least one of the following holds:
    \begin{enumerate}
        \item $\{x,y\}$ is not $\mathcal{R}_d$-linked in $G_2$, or
        \item $\{u,v\} \neq \{x,y\}$ and $\{x,y\}$ is not $\mathcal{R}_d$-linked in $G_1$.
    \end{enumerate}
    Then the pair $\{u,v\}$ is \strongly{d} in $G$ if and only if $\{u,v\}$ is \strongly{d} in $G_1$. 
\end{theorem}
\begin{proof}
    First, note that under our assumption, there is no $\mathcal{R}_d$-circuit in $G$ that intersects both $V_1 - V_2$ and $V_2 - V_1$. Indeed, if there existed such an $\mathcal{R}_d$-circuit $C$, then by \cref{lemma:circuitseparatingpair}, we would have $\{x,y\} \subseteq V(C)$. But then \cref{lemma:mconnected2sum} implies that $\{x,y\}$ is contained in an $\mathcal{R}_d$-circuit of $G_i+xy$, so the pair is $\mathcal{R}_d$-linked in $G_i$, for both $i \in \{1,2\}$, contradicting our assumption. 
    
    If $u$ and $v$ are adjacent in $G$, then the statement of the theorem is trivial, so let us assume that $u$ and $v$ are nonadjacent. If $\{u,v\}$ is \strongly{d} in $G_1$, then it is also \strongly{d} in $G$ by \cref{lemma:stronglylinkedsubgraph}. Let us thus assume that $\{u,v\}$ is \strongly{d} in $G$. In particular, $\{u,v\}$ is $\mathcal{R}_d$-linked in $G$, so there is an $\mathcal{R}_d$-circuit $C$ in $G+uv$ that contains $uv$. By the discussion in the first paragraph, we must have $V(C) \subseteq V(G_1)$. (Otherwise $V(C) \subseteq V(G_2)$, which would mean that $\{u,v\} = \{x,y\}$ and $\{x,y\}$ is $\mathcal{R}_d$-linked in $G_2$, contradicting our assumption.) It follows that $\{u,v\}$ is $\mathcal{R}_d$-linked in $G_1$.
    
    Let $(G,p)$ be a generic realization in $\CC^d$, and let $(G_i,p_i)$ be the subframework corresponding to $G_i$, for $i\in \{1,2\}$. By \cref{lemma:ksum}(c) we have $S(G,p) = S(G_1,p_1) \oplus S(G_2,p_2)$, where by a slight abuse of notation, we identify $S(G_i,p_i)$ with the stresses of $(G,p)$ supported on $E(G_i)$, for $i \in \{1,2\}$. Let $\omega$ be a stress of $(G_1+uv,p_1)$ that is supported on the edges of $C$. By \cref{lemma:def}(d), it suffices to show that $\omega$ is also a stress of $(G_1+uv,q_1)$ for every generic configuration $q_1 \in \sharedstresskernel{d}(G_1,p_1)$.

    Let us fix such a configuration $q_1$. By genericity, we have $S(G_1,p_1) = S(G_1,q_1)$. Let $(G,q)$ be the realization obtained from $(G_1,q_1)$ by gluing a suitably scaled and rotated copy of $(G_2,p_2)$ along $\{u,v\}$. Then by \cref{lemma:ksum}(c) again, we have \[S(G,q) = S(G_1,q_1) \oplus S(G_2,p_2) = S(G_1,p_1) \oplus S(G_2,p_2) = S(G,p),\] and hence $q \in \sharedstresskernel{d}(G,p)$. Since $\{u,v\}$ is \strongly{d} in $G$, it follows that $\omega$ is a stress of $(G+uv,q)$. But since $\omega$ is supported on the edges of $C$ and $E(C) \subseteq E(G_1)$, this implies that $\omega$ is also a stress of $(G_1+uv,q_1)$, as desired.
\end{proof}

\subsection{Stress-linked pairs in low dimensions}

In this subsection we give a combinatorial characterization of \strongly{d} vertex pairs in the $d=1$ and $d=2$ cases. Both characterizations involve the notion of \emph{local connectivity}.
Given a graph $G$ and a pair of nonadjacent vertices $u,v \in V(G)$, let $\kappa_G(u,v)$ denote the maximum number of pairwise internally disjoint $u,v$-paths in $G$. By Menger's theorem, this is the same as the minimum size of a vertex set $S$ such that $u$ and $v$ are disconnected in $G-S$.

The following (easy) theorem characterizes \strongly{1} pairs. %

\begin{theorem}\label{theorem:1dimchar} \emph{(Combinatorial characterization of \strongly{1} pairs.)}\newline
    Let $G$ be a graph and let $u,v \in V(G)$ be a pair of nonadjacent vertices. The following are equivalent.
    \begin{enumerate}
        \item $\{u,v\}$ is contained in a \globrig{1} subgraph of $G$.
        \item $\{u,v\}$ is \strongly{1} in $G$.
        \item $\{u,v\}$ is globally linked in $G$ in $\CC^1$.
        \item $\{u,v\}$ is globally linked in $G$ in $\RR^1$.
        \item $\kappa_G(u,v) \geq 2$.
    \end{enumerate}
\end{theorem}
\begin{proof}
    The first implication follows from \cref{lemma:globallyrigidstronglylinked,lemma:stronglylinkedsubgraph}, and the second from \cref{theorem:nongloballylinked}. The third and four implications are immediate. Finally, if $\kappa_G(u,v) \geq 2$, then there are two internally disjoint $u,v$-paths $P_1$ and $P_2$ in $G$. Now $P_1 \cup P_2$ is a $2$-connected (and hence \globrig{1}) subgraph that contains $\{u,v\}$.
\end{proof}

\cref{theorem:2dimchar} below characterizes \strongly{2} vertex pairs. As we noted in the introduction, Jackson, Jordán and Szabadka conjectured that the same condition characterizes globally linked pairs in $\RR^2$~\cite[Conjecture 5.9]{jackson.etal_2006}, and Jackson and Owen made the analogous conjecture for globally linked pairs in $\CC^2$~\cite[Conjecture 5.4]{jackson.owen_2019}.

\begin{theorem}\label{theorem:2dimchar} \emph{(Combinatorial characterization of \strongly{2} pairs.)}\newline
    Let $G = (V,E)$ be a graph and let $u,v \in V$ be a pair of vertices. The pair $\{u,v\}$ is \strongly{2} if and only if either $uv \in E$, or there is some $\mathcal{R}_2$-connected subgraph $H$ of $G$ with $u,v \in V(H)$ and $\kappa_H(u,v) \geq 3$. 
\end{theorem}
\begin{proof}
    We first prove sufficiency by induction on $|V|$. Our proof is essentially the same as that of~\cite[Theorem 5.7]{jackson.etal_2006} (which is the analogous statement for globally linked pairs in $\RR^2$). The cases when $|V| = 2$ or $uv \in E$ are trivial, so let us suppose that $|V| \geq 3$ and that $u$ and $v$ are nonadjacent in $G$. We need to show that $\{u,v\}$ is \strongly{2} in $G$. By \cref{lemma:stronglylinkedsubgraph}, it suffices to show that $\{u,v\}$ is \strongly{2} in $H$, so we may suppose that $G = H$. If $G$ is \globrig{2}, then we are done by \cref{lemma:globallyrigidstronglylinked}. Hence let us assume that this is not the case; since $G$ is $\mathcal{R}_2$-connected, \cref{theorem:2dimglobrigidchar}(b) now implies that $G$ is not $3$-connected. On the other hand, it can be deduced from \cref{lemma:circuitseparatingpair} that $G$ is $2$-connected. It follows that there is a separating vertex pair $\{x,y\}$ in $G$. We now have two cases: either $\{x,y\} = \{u,v\}$, or $\{x,y\} \neq \{u,v\}$.
    
    Let us first consider the case when $\{x,y\} = \{u,v\}$. Let $uw_1$ and $vw_2$ be edges of $G$ such that $w_1$ and $w_2$ are in different components of $G-\{u,v\}$. Since $G$ is $\mathcal{R}_2$-connected, there is an $\mathcal{R}_2$-circuit $C$ in $G$ that contains both $uw_1$ and $vw_2$. In particular, $\{u,v\}$ is a separating vertex pair of $C$. Let $C_1$ and $C_2$ be obtained by a $2$-separation of $C$ along $\{u,v\}$.  \cref{lemma:mconnected2sum} implies that $C_i$ is an $\mathcal{R}_2$-circuit, and thus $\{u,v\}$ is $\mathcal{R}_2$-linked in $C_i-uv$, for $i \in \{1,2\}$. It follows from \cref{theorem:gluing} that $\{u,v\}$ is \strongly{2} in $C$, and consequently in $G$ as well by \cref{lemma:stronglylinkedsubgraph}.
    
    Now let us consider the case when $\{x,y\} \neq \{u,v\}$. By the previous reasoning, we still get that $\{x,y\}$ is \strongly{2} in $G$, so we may assume that $xy \in E$. Let $G_1$ and $G_2$ be obtained by a $2$-separation of $G-xy$ along $\{x,y\}$, with $\{u,v\} \subseteq V(G_1)$. \cref{lemma:mconnected2sum} implies that $G_1$ is also $\mathcal{R}_2$-connected, and since $xy \in E(G_1)$ we also have $\kappa_{G_1}(u,v) \geq 3$. Thus by induction, $\{u,v\}$ is \strongly{2} in $G_1$, and hence by \cref{lemma:stronglylinkedsubgraph}, it is also \strongly{2} in $G$.
    
    Next, we prove necessity by induction on $|E|$. The case when $|E| = 1$ or $uv \in E$ is immediate, and by \cref{theorem:2dimglobrigidchar}, so is the case when $G$ is \globrig{2}. If there is an $\mathcal{R}_2$-bridge $e$ in $G$, then $\{u,v\}$ is also \strongly{2} in $G-e$ by \cref{theorem:bridgegloballylinked}, and thus we are done by induction. Hence we may assume that $G$ is not \globrig{2} and that it does not contain $\mathcal{R}_2$-bridges. From \cref{theorem:2dimglobrigidchar}(c) it follows that $G$ is not $3$-connected. 
    
    Note that since $\{u,v\}$ is \strongly{2}, it is also globally linked in $\RR^2$ by \cref{theorem:nongloballylinked}, which implies that $\kappa_G(u,v) \geq 3$.
    If $G$ is not $2$-connected or $G$ has an adjacent separating pair of vertices, then we can use \cref{lemma:completeseparator} and induction. Thus we may assume that there is a separating pair of vertices $\{x,y\}$ in $G$ such that $xy \notin E$. Let $G_1$ and $G_2$ be induced subgraphs of $G$ such that $G = G_1 \cup G_2$, $V(G_1) \cap V(G_2) = \{x,y\}$ and $u,v \in V(G_1)$.
    
    If $\{x,y\}$ is not $\mathcal{R}_d$-linked in some of $G_1$ or $G_2$, then by \cref{theorem:2separatorconverse}, $\{u,v\}$ is \strongly{2} in either $G_1$ or $G_2$, so we are done by induction. Hence we may assume that $\{x,y\}$ is $\mathcal{R}_d$-linked in both $G_1$ and $G_2$. If $\{u,v\} = \{x,y\}$, then we can take an $\mathcal{R}_2$-circuit $C_i$ in $G_i+uv$ containing $uv$ such that $V(C_i) \subseteq V(G_i)$, for each $i \in \{1,2\}$. By \cref{lemma:mconnected2sum}, $C_1 \cup C_2 - uv$ is an $\mathcal{R}_2$-circuit with $\kappa_C(u,v) \geq 3$.\footnote{The observation that $\kappa_C(u,v) \geq 3$ follows from the well-known fact that $\mathcal{R}_2$-circuits are redundantly \rig{2}, and in particular redundantly $2$-connected. This implies that $\kappa_{C_i - uv}(u,v) \geq 2$, and hence $\kappa_C(u,v) \geq 4$.}
    
    Finally, suppose that $\{u,v\} \neq \{x,y\}$. Since $\{u,v\}$ is also \strongly{2} in $G+xy$, it is \strongly{2} in $G_1+xy$ by \cref{lemma:completeseparator}. By induction, there is an $\mathcal{R}_2$-connected subgraph $H$ of $G_1 + xy$ with $u,v \in V(H)$ and $\kappa_H(u,v) \geq 3$. If $xy \notin E(H)$, then $H$ is also a subgraph of $G$, and we are done. Otherwise, we can take an $\mathcal{R}_2$-circuit $C_2$ of $G_2+xy$ with $xy \in E(C_2)$ and consider the graph $H' = (H \cup C_2) - xy$. By \cref{lemma:mconnected2sum}, $H'$ is also $\mathcal{R}_2$-connected, and clearly $\kappa_{H'}(u,v) \geq 3$ also holds.
\end{proof}

\subsection{Stress-linked pairs and coning}

Given a graph $G = (V,E)$, the \emph{cone graph} of $G$ (with cone vertex $c$) is the graph $\coneG$ obtained from $G$ by adding a new vertex $c$ and new edges $cv, v \in V$. Coning is a well-studied operation in combinatorial rigidity theory that often provides a transfer of rigidity properties from dimension $d$ to dimension $d+1$. For example, a classic result of Whiteley \cite{whiteley_cones} states that $G$ is $d$-rigid if and only if $\coneG$ is $(d+1)$-rigid. Similarly, by a result of Connelly and Whiteley \cite{connelly.whiteley_2010}, a graph is globally $d$-rigid if and only if $\coneG$ is globally $(d+1)$-rigid. In this subsection we prove the following refinement of the latter result.

\begin{theorem}\label{theorem:stresslinkedconing} \emph{(Coning and stress-linked pairs.)} \newline
    Let $\coneG$ be the cone of the graph $G = (V,E)$, and let $u,v \in V$ be a pair of vertices. The pair $\{u,v\}$ is \strongly{d} in $G$ if and only if it is \strongly{(d+1)} in $\coneG$.
\end{theorem}

To prove \cref{theorem:stresslinkedconing}, we need to understand the shared stress nullity $\stressnullity{G}$ and the property of being $\mathcal{R}_d$-linked behave under coning. For the latter, we have the following folklore result.

\begin{lemma}\label{lemma:linkedconing} \emph{(Coning and linked pairs.)} \newline
    Let $\coneG$ be the cone of the graph $G = (V,E)$, and let $u,v \in V$ be a pair of nonadjacent vertices. The pair $\{u,v\}$ is $\mathcal{R}_d$-linked in $G$ if and only if it is $\mathcal{R}_{d+1}$-linked in $\coneG$. 
\end{lemma}
\begin{proof}
    This is immediate from the fact that $r_d(G') = r_d(G) + |V|$ and $r_d(G' + uv) = r_d(G+uv) + |V|$, which are consequences of Whiteley's coning theorem~\cite{whiteley_cones} (e.g., see~\cite[Theorem 5.2]{garamvolgyi.etal_2022}).
\end{proof}

To understand how coning affects spaces of stresses, we first investigate the operation of sliding.
Let $\coneG$ be the cone of the graph $G = (V,E)$ with cone vertex $c$. Given a configuration $p \in (\CC^d)^{V(G')}$ and a vector of nonzero scalars $\lambda = {(\lambda_v)}_{v \in V} \in \CC^V$, we define the configuration $p_\lambda \in (\CC^d)^{V(G')}$ by letting $p_\lambda(c) = p(c)$ and $p_\lambda(v) = p(c) + \lambda_v\big(p(v) - p(c)\big)$ for each $v \in V$. We say that $p_\lambda$ is obtained from $p$ by \emph{sliding} (according to $\lambda$). The following folkore result shows that stress spaces behave well under this operation.

\begin{lemma}\label{lemma:sliding} \emph{(Stress spaces and sliding.)} \newline
    Let $\coneG$ be the cone of the graph $G = (V,E)$ with cone vertex $c$. Fix a realization $(\coneG,p)$ of $\coneG$ in $\CC^d$, let $\lambda \in \CC^V$ be a vector of nonzero scalars, and consider the realization $(\coneG,p_\lambda)$ obtained from $(\coneG,p)$ by sliding according to $\lambda$.
    \begin{enumerate}
        \item There is a linear isomorphism $\Phi_\lambda : \CC^{E(\coneG)} \to \CC^{E(\coneG)}$, depending only on $\lambda$, such that $S(\coneG,p_\lambda) = \Phi_\lambda(S(\coneG,p))$. Moreover, we have $\Phi_\lambda \circ \Phi_{\mu} = \text{id}$, where $\mu = {(1/\lambda_v)}_{v \in V}  \in \CC^V$. 
        \item We have $K_1(\coneG,p_\lambda) = \{q_\lambda : q \in K_1(\coneG,p)\}$. In particular, $K_1(\coneG,p)$ and $K_1(\coneG,p_\lambda)$ are linearly isomorphic.
    \end{enumerate}
\end{lemma}
\begin{proof}
    \emph{(a)} This is implicit in~\cite[Lemma 4.9]{connelly.etal_2017}. Explicitly, $\Phi_\lambda$ is defined by 
    \[{\Phi_\lambda(\omega)}_{uv} = 
        \begin{cases}
        \dfrac{\omega_{uv}}{\lambda_u \lambda_v} &\text{if } u,v \in V, \\[10pt] \dfrac{\omega_{cv}}{\lambda_v^2} + \dfrac{1}{\lambda_v}{\displaystyle \sum_{vx \in E}}\omega_{vx}\left(\dfrac{1}{\lambda_v} - \dfrac{1}{\lambda_x}\right) &\text{if } v \in V \text{ and } u = c. 
        \end{cases}
    \]
    The claimed properties of $\Phi_\lambda$ can be checked by direct calculation. For a detailed exposition, see~\cite{martinwinter}.

    \emph{(b)} Fix $q \in K_1(\coneG,p)$. By definition, $S(\coneG,p) \subseteq S(\coneG,q)$ and so by part \emph{(a)}, we have \[S(\coneG,p_\lambda) = \Phi_\lambda\left(S(\coneG,p)\right) \subseteq \Phi_\lambda\left(S(\coneG,q)\right) = S(\coneG,q_\lambda).\] Hence $q_\lambda \in K_1(\coneG,p_\lambda)$. Conversely, every member of $K_1(\coneG,p_\lambda)$ may be written as $q_\lambda$ for some realization $(\coneG,q)$, which implies that \[S(\coneG,p) = \Phi_\mu(S(\coneG,p_\lambda)) \subseteq \Phi_\mu(S(\coneG,q_\lambda)) = S(\coneG,q),\] so $q \in K_1(\coneG,p)$, as desired. 
    
    This shows that $K_1(\coneG,p_\lambda) = \{q_\lambda : q \in K_1(\coneG,p)\}$. It follows immediately from the definition of sliding that the spaces $K_1(\coneG,p)$ and $K_1(\coneG,p_\lambda)$ are linearly isomorphic. 
\end{proof}

\begin{proposition}\label{proposition:stressnullityconing} \emph{(Coning and the shared stress nullity.)} \newline
    Let $\coneG$ be the cone of the graph $G = (V,E)$ with cone vertex $c$. We have $\stressnullitydim{\coneG}{d+1} = \stressnullity{G}$ + 1.
\end{proposition}
\begin{proof}
    Let us fix a generic realization $(G,p_0)$ in $\CC^d$. Let $(\coneG,p)$ be the realization in $\CC^{d+1}$ obtained by setting $p(c) = 0$ and $p(v) = {(p_0(v),1)}^T \in \CC^{d+1}$ for each $v \in V$. Finally, fix a collection of nonzero scalars $\lambda = {(\lambda_v)}_{v \in V} \in \CC^V$ such that the coordinates in $\{\lambda\} \cup \{p_0(v) : v \in V\}$ form an algebraically independent set, and consider the framework $(\coneG,p_\lambda)$ obtained from $(\coneG,p)$ by sliding.

    Since $(G,p_0)$ is generic, we have $\dim(K_1(G,p_0)) = \stressnullity{G}$. It is routine to check that we can obtain a generic realization by suitably translating $(\coneG,p_\lambda)$, which implies that $\dim(K_1(\coneG,p_\lambda)) = \stressnullitydim{\coneG}{d+1}$. Hence, we only need to show that $\dim(K_1(\coneG,p_\lambda)) = \dim(K_1(G,p_0)) + 1$. 
    
    By \cref{lemma:sliding}(b), $K_1(\coneG,p_\lambda)$ is linearly isomorphic to $K_1(\coneG,p)$. Let us analyze the structure of this latter set. First, observe that \[S(\coneG,p) = \{\omega \in \CC^{E(\coneG)}: \omega|_{E} \in S(G,p_0), \, \omega_{cv} = 0 \, \, \forall v \in V\},\] that is, the stresses of $(\coneG,p)$ are precisely the stresses of $(G,p)$ with additional zeros on the cone edges. Indeed, for any stress $\omega \in S(\coneG,p)$ and any $v \in V$, we must have $\omega_{cv} = 0$, since otherwise the equilibrium conditions could not be satisfied at $v$, due to the fact that in the framework $(\coneG,p)$, the rest of the edges incident to $v$ lie in a hyperplane. 
    
    It follows from the definitions that $K_1(\coneG,p) = \{(\coneG,q) : (G,q|_V) \in K_1(G,p_0)\}$. Thus we have $\dim(K_1(\coneG,p)) = \dim(K_1(G,p_0)) + 1$, where the plus one comes from the fact that there is no constraint in $K_1(\coneG,p)$ on the position of cone vertex $v$. Hence \[\stressnullitydim{\coneG}{d+1} = \dim(K_1(\coneG,p_\lambda)) = \dim(K_1(\coneG,p)) = \dim(K_1(G,p_0))+1 = \stressnullity{G} +1,\]as claimed.
\end{proof}

We can now easily deduce the main result of this subsection.

\begin{proof}[Proof of \cref{theorem:stresslinkedconing}]
    The statement follows immediately from \cref{lemma:linkedconing}, \cref{proposition:stressnullityconing}, and the definition of $d$-stress-linked vertex pairs.
\end{proof}

As a simple application of \cref{theorem:stresslinkedconing}, we show that if a pair of vertices is \strongly{d}, then the same pair is \strongly{d'} for all $d' \leq d$. While this result is unsurprising, we note that it appears to be open whether an analogous result holds for globally linked pairs, either in $\RR^d$ or in $\CC^d$.

\begin{corollary}\label{corollary:dimensiondropping} \emph{(Dimension dropping for stress-linked pairs.)} \newline
    Let $G$ be a graph and let $u,v \in V(G)$ be a pair of vertices. If $\{u,v\}$ is \strongly{d} in $G$, then $\{u,v\}$ is \strongly{d'} in $G$ for all nonnegative integers $d' \leq d$.
\end{corollary}
\begin{proof}
    If $\{u,v\}$ is \strongly{d} in $G$, then by \cref{lemma:stronglylinkedsubgraph}, $\{u,v\}$ is also \strongly{d} in the $(d-d')$-fold cone of $G$, and hence by \cref{theorem:stresslinkedconing}, $\{u,v\}$ is \strongly{d'} in $G$, as claimed.
\end{proof}

\section{Applications}\label{section:applications}

In this section we show how our results on stress-linked vertex pairs can be applied to various problems in rigidity theory. We start by significantly strengthening the main result of~\cite{garamvolgyi.jordan_2023}.

A graph $G$ is \emph{minimally \globrig{d}} if it is \globrig{d}, but $G-e$ is not \globrig{d}, for every edge $e \in E(G)$. In~\cite{garamvolgyi.jordan_2023}, Jordán and the author proved that a minimally \globrig{d} graph on $n \geq d+2$ vertices has at most $(d+1)n - \binom{d+2}{2}$ edges, and that the inequality is strict unless $G$ is the complete graph on $d+2$ vertices. We shall now prove that the same sparsity bound holds for every subgraph of $G$ on at least $d+2$ vertices. We will need the following result about the rigidity matroid of \globrig{d} graphs. 

\begin{theorem} (\cite[Theorem 3.5]{garamvolgyi.etal_2022})\label{theorem:globallyrigidMconnected}
    If $G$ is a \globrig{d} graph on at least $d+2$ vertices, then it is $\mathcal{R}_d$-connected.
\end{theorem}

\noindent The uniform sparsity of minimally globally $d$-rigid graphs follows easily from the following general result. 

\begin{theorem}\label{theorem:sparsity} \emph{(Sparsity and stress-linked pairs.)} \newline
    Let $G = (V,E)$ be a graph on $n \geq d+2$ vertices. Suppose that for every edge $uv \in E$, $\{u,v\}$ is not \strongly{d} in $G-uv$. Then $|E| \leq r_d(G) + n - d - 1$, with equality if and only if $G = K_{d+2}$. 
\end{theorem}
\begin{proof}
    Let $G_0 = (V,E_0)$ be a maximal $\mathcal{R}_d$-independent spanning subgraph of $G$, let $E - E_0 = \{e_1,\ldots,e_k\}$, and let us define $G_i = G_0 + \{e_1,\ldots,e_i\}$ for each $i \in \{1,\ldots,k\}$. The maximal choice of $G_0$ and the assumption on $G$ imply that the pair of end vertices $\{u_i,v_i\}$ of $e_i$ is $\mathcal{R}_d$-linked, but not \strongly{d} in $G_{i-1}$, for each $i \in \{1,\ldots,k\}$. Hence we have
\[d+1 \leq \stressnullity{G} = \stressnullity{G_k} < \stressnullity{G_{k-1}} < \cdots < \stressnullity{G_1} < \stressnullity{G_0} \leq n.\]It follows that $k \leq n-d-1$ and \[|E| = |E_0| + k \leq r_d(G) + n-d-1.\]
    
    It only remains to characterize the cases of equality. Let us thus assume that $|E| = r_d(G) + n - d - 1$.  
    The argument in the first paragraph now gives $\stressnullity{G} = d+1$, so by \cref{theorem:ght}, $G$ is \globrig{d}. 
    
    We first claim that every $\mathcal{R}_d$-circuit in $G$ is a copy of $K_{d+2}$. Let us suppose, for a contradiction, that this is not the case, and let $C$ be an $\mathcal{R}_d$-circuit in $G$ that is not isomorphic to $K_{d+2}$. Let $e \in E(C)$ be an arbitrary edge of $C$. Now $C-e$ is $\mathcal{R}_d$-independent, and hence we may extend it to a maximal $\mathcal{R}_d$-independent subgraph $G_0' = (V,E_0')$. Note that $G_1' = G_0' + e$ contains the circuit $C$, so by \cref{lemma:largenullity} we have $\stressnullity{G_1} \leq n-2$. We can now repeat the argument of the first paragraph with $G_0'$ and starting with the edge $e_1 = e$ to deduce that $|E- E_0| \leq n-d-2$, and thus $|E| < r_d(G) + n-d-1$, contradicting our assumption on $G$.
    
    The rest of the proof follows that of~\cite[Theorem 3.7]{garamvolgyi.jordan_2023}. By \cref{theorem:globallyrigidMconnected}, $G$ is $\mathcal{R}_d$-connected, and thus every pair of edges in $G$ is contained in an $\mathcal{R}_d$-circuit. Since every $\mathcal{R}_d$-circuit of $G$ is complete, this means that every pair of edges in $G$ is contained in a complete subgraph. But a graph with this property and without isolated vertices must  necessarily be complete: if $u$ and $v$ are arbitrary vertices and $e$ and $f$ are edges incident to $u$ and $v$, respectively, then the fact that $e$ and $f$ are contained in a complete subgraph shows that $u$ and $v$ must be adjacent in $G$.
    
    Thus $G$ is a complete graph on at least $d+2$ vertices. It is well-known that if $n \geq d+3$, then $G-uv$ is \globrig{d} for any edge $uv$, and thus by \cref{lemma:globallyrigidstronglylinked}, $\{u,v\}$ is \strongly{d} in $G-uv$, contradicting our assumption on $G$. It follows that $G = K_{d+2}$, as claimed.
\end{proof}

\begin{theorem}\label{corollary:minglobrigid} \emph{(Sparsity of minimally globally rigid graphs.)} \newline
    Let $G$ be a graph. If $G$ is minimally \globrig{d}, then for each subset $X \subseteq V(G)$ with $|X| \geq d+2$ we have \[|E(G[X])| \leq r_d(G[X]) + |X| - d - 1 \leq (d+1)|X| - \binom{d+2}{2}.\] The first inequality is strict unless $G[X] = K_{d+2}$.
\end{theorem}
\begin{proof}
    Let us fix $X \subseteq V(G)$. If $|X| = d+2$, then it is easy to see that the first inequality is satisfied and that it holds with equality if and only if $G[X] = K_{d+2}$. Thus let us suppose that $|X| \geq d+3$, and let us consider a pair of vertices $u,v \in X$ with $uv \in E(G)$. Since $G$ is minimally \globrig{d}, $\{u,v\}$ is not globally linked in $G-uv$ in $\CC^d$, and hence it is not globally linked in $G[X] - uv$ in $\CC^d$. By \cref{theorem:nongloballylinked} it follows that $\{u,v\}$ is not \strongly{d} in $G[X]-uv$. Thus \cref{theorem:sparsity} applies, giving $|E(G[X])| < r_d(G[X]) + |X| - d - 1$. The second inequality follows from the bound $r_d(G[X]) \leq d|X| - \binom{d+1}{2}$ by a straightforward calculation.
\end{proof}

Using the coning results presented in the previous section, we can also obtain sparsity results for minimally edge- and vertex-redundantly globally rigid graphs.
Let $G = (V,E)$ be a graph and let $k,\ell$ be a pair of positive integers. We say that $G$ is \emph{globally $[k,\ell,d]$-rigid} if $|V| \geq k$ and $G - S - F$ is globally $d$-rigid for any sets $S \subseteq V$ and $F \subseteq E$ satisfying $|S| \leq k-1$ and $|F| \leq \ell-1$. Using this notation, being globally $d$-rigid is the same as being globally $[1,1,d]$-rigid.
We say that $G$ is \emph{minimally globally $[k,\ell,d]$-rigid} if it is globally $[k,\ell,d]$-rigid, but $G-e$ is not, for every $e \in E$.

\begin{theorem} \emph{(Sparsity of minimally redundantly globally rigid graphs.)} \newline
    Let $G = (V,E)$ be a minimally globally $[k,\ell,d]$-rigid graph and let $\bigdim = d+k+\ell- 2$. For each subset $X \subseteq V(G)$ with $|X| \geq \bigdim + 2$, we have
    \begin{align*}
        |E(G[X])| \leq r_{\bigdim}(G[X]) + |X| - \bigdim - 1 \leq (\bigdim + 1)|X| - \binom{\bigdim+2}{2}.
    \end{align*}
    The first inequality is strict unless $G[X] = K_{\bigdim+2}$.
\end{theorem}
\begin{proof}
    We will show that for every edge $uv \in E$, $\{u,v\}$ is not \strongly{\bigdim} in $G-uv$. The result then follows from \cref{theorem:sparsity}.

    Let us fix $uv \in E(G)$. By the minimality of $G$, $G-uv$ is not globally $[k,\ell,d]$-rigid, and thus there exist sets $S \subseteq V$ and $F \subseteq E$ with $|S| \leq k-1$ and $|F| \leq \ell-1$ and such that $G-S-F-uv$ is not globally $d$-rigid. By assumption, $G-S-F$ is globally $d$-rigid; in particular, $uv \notin F$, $u,v \notin S$, and by \cref{lemma:stronglylinkedsubgraph}, $\{u,v\}$ is not \strongly{d} in $G-S-F-uv$.

    Let us write $S = \{s_1,\ldots,s_q$\} and $F = \{u_1v_1,\ldots,u_tv_t\}$. We may assume that $u,v \notin \{u_1,\ldots,u_t\}$ by possibly swapping the label of $u_i$ and $v_i$. Let $C = \{s_1,\ldots,s_q,u_1,\ldots,u_t\}$ and let $G' = G - uv + \{xy: x \in V, y \in C\}$. Note that $G'$ is isomorphic to the $|C|$-fold cone of $G-C-uv$, and that it contains $G-uv$ as a subgraph. Since $\{u,v\}$ is not \strongly{d} in $G-S-F-uv$, by \cref{lemma:stronglylinkedsubgraph} it is not \strongly{d} in $G-C-uv$. If follows from \cref{theorem:stresslinkedconing} that $\{u,v\}$ is not \strongly{(d+|C|)} in $G'$, and hence by \cref{corollary:dimensiondropping} (and the fact that $d+|C| \leq \bigdim$) it is not \strongly{\bigdim} in $G'$. We conclude by \cref{lemma:stronglylinkedsubgraph} that $\{u,v\}$ is not \strongly{\bigdim} in $G-uv$, as desired.
\end{proof}

We can use
\cref{corollary:bridges} and \cref{theorem:gluing} to verify two conjectures that are unrelated to the notion of globally linked pairs or global rigidity. Both of them follow quickly from the following gluing result. We say that a graph $G$ is \emph{$\mathcal{R}_d$-bridgeless} if it does not contain any $\mathcal{R}_d$-bridges.

\begin{theorem}\label{theorem:Rdbridgeless}\emph{($\mathcal{R}_d$-bridgeless gluing theorem.)}\newline
     Let $G$ be the union of the graphs $G_1 = (V_1,E_1), G_2 = (V_2,E_2)$ with $|V_1 \cap V_2| \leq d+1$, and let $uv \in E_1 \cap E_2$. Suppose that $G$ is $\mathcal{R}_d$-bridgeless and that $\{u,v\}$ is $\mathcal{R}_d$-linked in both $G_1-uv$ and $G_2-uv$. Then $G-uv$ is also $\mathcal{R}_d$-bridgeless.
\end{theorem}
\begin{proof}
    By \cref{theorem:gluing}, $\{u,v\}$ is \strongly{d} in $G-uv$, and hence by \cref{corollary:bridges}, the set of $\mathcal{R}_d$-bridges is the same in $G-uv$ and in $G$. Hence if one is $\mathcal{R}_d$-bridgeless, then so is the other. 
\end{proof}

In~\cite{connelly_2011a}, Connelly showed that if $G = G_1 \cup G_2$, where $G_1$ and $G_2$ are \globrig{d}, $|V(G_1) \cap V(G_2)| = d+1$, and $uv \in E(G_1) \cap E(G_2)$, then $G-uv$ is also \globrig{d}. He asked whether a similar ``gluing'' result holds for redundantly \rig{d} graphs, that is, graphs that remain \rig{d} after the deletion of any edge. Note that a graph is redundantly \rig{d} if and only if it is \rig{d} and $\mathcal{R}_d$-bridgeless. Our next result gives an affirmative answer to this question. We note that the same proof technique can also be used to give a new proof of the result of Connelly.

\begin{theorem}\label{theorem:connellyquestion} \emph{(Redundantly rigid gluing theorem.)} \newline
    Let $G$ be the union of the graphs $G_1 = (V_1,E_1), G_2 = (V_2,E_2)$ with $|V_1 \cap V_2|\in \{d,d+1\}$, where $G_1$ and $G_2$ are both redundantly \rig{d}. Then $G-uv$ is redundantly \rig{d} for any $uv \in E_1 \cap E_2$. 
\end{theorem}
\begin{proof}
    Let us fix $uv \in E_1 \cap E_2$. It follows from \cref{lemma:whiteleygluing}(a) that $G$ is redundantly \rig{d}, and hence $G-uv$ is \rig{d}. Since $G_1$ and $G_2$ are both $\mathcal{R}_d$-bridgeless, so is $G$, and $\{u,v\}$ is $\mathcal{R}_d$-linked in both $G_1-uv$ and $G_2-uv$. By \cref{theorem:Rdbridgeless}, we conclude that $G-uv$ is $\mathcal{R}_d$-bridgeless, as desired.
\end{proof}

We say that a graph $G$ is the \emph{$t$-sum} of two graphs $G_1$ and $G_2$ along an edge $e \in E(G_1) \cap E(G_2)$ if $G = G_1 \cup G_2 -e$ and $G_1 \cap G_2$ is a complete graph on $t$ vertices. 
In~\cite{grasegger.etal_2022}, Grasegger, Guler, Jackson and Nixon conjectured that if $2 \leq t \leq d+1$ and $G$ is the $t$-sum of $G_1$ and $G_2$, then $G$ is an $\mathcal{R}_d$-circuit if and only if $G_1$ and $G_2$ are both $\mathcal{R}_d$-circuits. Subsequently, in~\cite{grasegger.etal_2023}, they gave a counterexample to the ``only if'' direction. Here we show that the ``if'' direction holds. In fact, we give a full characterization of when $G$ is an $\mathcal{R}_d$-circuit, as follows.

\begin{theorem}\label{theorem:tsumcircuitstronger} \emph{($\mathcal{R}_d$-circuits and $t$-sums.)} \newline
    Let $G$ be the $t$-sum of the graphs $G_1 = (V_1,E_1)$ and $G_2 = (V_2,E_2)$ along an edge $uv \in E_1 \cap E_2$, for some integer $t$ with $2 \leq t \leq d+1$. Then $G$ is an $\mathcal{R}_d$-circuit if and only $G_i$ contains a unique $\mathcal{R}_d$-circuit $G'_i$ for $i \in \{1,2\}$, and we have $uv \in E(G'_1) \cap E(G'_2)$ and $G+uv = G_1' \cup G_2'$.
\end{theorem}
\begin{proof}
    Necessity is shown in~\cite[Lemma 1]{grasegger.etal_2023}, so we only prove sufficiency. It follows from \cref{lemma:ksum}(b) that each generic realization of $G$ in $\CC^d$ has a one-dimensional space of stresses, i.e., that $G$ contains a unique $\mathcal{R}_d$-circuit. Thus it suffices to show that $G$ is $\mathcal{R}_d$-bridgeless. This follows from applying \cref{theorem:Rdbridgeless} to the decomposition $G+uv = G_1' \cup G_2'$. 
\end{proof}

\begin{corollary}
    For any integer $t$ with $2 \leq t \leq d+1$, the $t$-sum of two $\mathcal{R}_d$-circuits is again an $\mathcal{R}_d$-circuit.
\end{corollary}

\section{Concluding remarks}\label{section:concluding}

In this paper we introduced (among others) the notion of \strongly{d} vertex pairs. We showed that \strongly{d} pairs are also globally linked in $\CC^d$, and we proved the ``\strongly{d}'' analog of a number of conjectures about globally linked pairs. The biggest remaining question is whether being \strongly{d} is actually equivalent to being globally linked in $\RR^d$ (or in $\CC^d$). We pose the stronger version as a conjecture.

\begin{conjecture}\label{conjecture:stresskernel}
    Let $G$ be a graph and let $u,v \in V(G)$ be a pair of vertices. The pair $\{u,v\}$ is globally linked in $G$ in $\RR^d$ if and only if  $\{u,v\}$ is \strongly{d} in $G$.
\end{conjecture}

To support \cref{conjecture:stresskernel}, let us note that globally linked pairs in $\CC^d$ have a characterization that is analogous to the following equivalent definition of \strongly{d} pairs (cf.\ \cref{lemma:def}(b)): a vertex pair $\{u,v\}$ is \strongly{d} in $G$ if and only if it is $\mathcal{R}_d$-linked in $G$ and for any generic realization $(G,p)$ in $\CC^d$, we have $\sharedstresskernel{d}(G,p) \subseteq \sharedstresskernel{d}(G+uv,p)$.

\begin{proposition}\label{proposition:globallylinkedchar}
    Let $G$ be a graph and let $u,v \in V(G)$ be a pair of vertices. Let $(G,p)$ be a generic realization in $\CC^d$, and let $R_p = m_{d,G}^{-1}(m_{d,G}(p))$ denote the set of configurations $q$ for which $(G,q)$ is equivalent to $(G,p)$. Then $\{u,v\}$ is globally linked in $G$ in $\CC^d$ if and only if $\{u,v\}$ is $\mathcal{R}_d$-linked in $G$ and $R_p \subseteq \sharedstresskernel{d}(G+uv,p)$.
\end{proposition}
\begin{proof}[Proof (sketch)]
    If $\{u,v\}$ is globally linked in $G$ in $\CC^d$, then $R_p$ is in fact equal to the fiber $m_{d,G+uv}^{-1}(m_{d,G+uv}(p))$, which is contained in $\sharedstresskernel{d}(G+uv,p)$ by \cref{theorem:connelly}. The other direction follows from the proof of \cref{theorem:nongloballylinked}.
\end{proof}

We note that checking whether a pair of vertices is \strongly{d} can be done in randomized polynomial time with an approach similar to the one given in~\cite[Section 5]{gortler.etal_2010} for checking global $d$-rigidity. In the $d = 1$ and $d = 2$ cases, \cref{theorem:1dimchar,theorem:2dimchar} also lead to efficient deterministic recognition algorithms. \cref{conjecture:stresskernel} would imply the analogous algorithmic results for recognizing globally linked vertex pairs.

It would be natural to try to adapt the proof of \cref{theorem:ght} in~\cite{gortler.etal_2010} to prove \cref{conjecture:stresskernel}. One sign of difficulty for this approach is that being globally linked in $\RR^d$ is not a generic property.
The following conjecture may remedy this problem to some extent. It says that if $\{u,v\}$ is not globally linked in $G$ in $\RR^d$, then $\sharedstresskernel{d}(G,p)$ contains a witness of this, for every generic realization $(G,p)$ in $\RR^d$.

\begin{conjecture}
    Let $G$ be a graph and let $u,v \in V(G)$ be a pair of vertices. Suppose that $\{u,v\}$ is not globally linked in $G$ in $\RR^d$. Then for every generic realization $(G,p)$ in $\RR^d$, there exist real configurations $q,q' \in \sharedstresskernel{d}(G,p)$ such that $(G,q)$ is generic and $(G,q)$ and $(G,q')$ are equivalent, but $(G+uv,q)$ and $(G+uv,q')$ are not equivalent.
\end{conjecture}

One of the most interesting features of \strongly{d} vertex pairs is given by \cref{corollary:bridges}, which says that adding an edge between a \strongly{d} vertex pair in a graph does not change the set of its $\mathcal{R}_d$-bridges. The following conjecture is a generalization of this fact from $\mathcal{R}_d$-bridges to $\mathcal{R}_d$-components.

\begin{conjecture}\label{conjecture:Mcomp}
    Let $G$ be a graph and let $u,v \in V(G)$ be a pair of vertices. If $\{u,v\}$ is \strongly{d} in $G$, then there is some $\mathcal{R}_d$-component $G_0$ of $G$ such that $\{u,v\}$ is \strongly{d} in $G_0$.
\end{conjecture}

In particular, \cref{conjecture:Mcomp} implies that if $\{u,v\}$ is \strongly{d} in $G$, then the $\mathcal{R}_d$-components of $G$ and $G+uv$ are essentially the same: the addition of $uv$ does not ``merge'' any two of these components. Similarly to how \cref{corollary:bridges} can be seen as a refinement of Hendrickson's theorem, \cref{conjecture:Mcomp} would give a refinement of \cref{theorem:globallyrigidMconnected}. Note that \cref{conjecture:Mcomp} holds in the $d = 1$ and $d = 2$ cases by \cref{theorem:1dimchar,theorem:2dimchar}.

In closing, let us briefly consider the applications of the ideas presented in this paper in a wider context. Rigidity theory is mainly concerned with the study of the homogeneous polynomial map $m_{d,G} : (\CC^{d})^V \rightarrow \CC^E$. A key feature of this map is that it is quadratic, and hence the entries of its Jacobian matrix are linear polynomials. This implies that for any vector $\omega \in \CC^E$, 
the stress kernel $\stresskernel{d}{G}(\omega)$ is linear, and in particular it is an irreducible affine variety. Similarly, the shared stress kernel $\sharedstresskernel{d}(G,p)$ of any configuration $p \in (\CC^{d})^V$ is linear and hence irreducible.

There are other natural problems that can be phrased as questions about an algebraic matroid given by the image of some homogeneous polynomial map $f$. Often $f$ maps some configuration space into $\RR^E$ or $\CC^E$, where $E$ is the edge set of a  (hyper-)graph, see~\cite{cruickshank.etal_2023,rosen.etal_2020}. For each such problem, the analog of the shared stress kernel can be defined, and whenever this space is irreducible, at least for generic points of the configuration space, we can also define an analog of stress-linked vertex pairs (or more generally, stress-linked hyperedges). 

This is the case, for example, in the low-rank matrix completion problem, where the associated polynomial map is given by the standard dot product and hence is quadratic, see~\cite{kiraly.etal_2015,singer.cucuringu_2010}. Applying the methods of this paper to this problem leads to an analog of \cref{theorem:sparsity}, giving a linear upper bound on the number of known entries in a minimally generically uniquely rank-$r$ completable partial matrix over $\RR$. Note that, in contrast with graph rigidity, in the case of matrix completion ``being uniquely completable over $\RR$'' is not a generic property (and thus is not the same as ``being uniquely completable over $\CC$''), see~\cite{jackson.etal_2016}. It would be interesting to understand the structure of the associated ``generic stress matroid'' in this case.

\section*{Acknowledgements}

I would like to thank Soma Villányi for the insightful question that led to the definition of \strongly{d} vertex pairs and, ultimately, to the results presented in this paper. I would also like to thank Steven Gortler and Tibor Jordán for useful discussions. Finally, I would like to thank Tony Nixon for the suggestion to also investigate the case of equality in \cref{theorem:sparsity}.

Parts of this paper were written while I attended the ``Focus Program on Geometric Constraint Systems'' at the Fields Institute for
Research in Mathematical Sciences, Toronto, and the ``Geometry of Materials, Packings and Rigid Frameworks'' semester program at ICERM, Providence. I am grateful to both the Fields Institute and ICERM for providing an excellent working environment.

This research was supported by the Hungarian Scientific Research Fund
grant no.\ K135421.

\printbibliography

\appendix
\crefalias{subsection}{appsec}
\section*{Appendix}
\addcontentsline{toc}{section}{Appendix}
\renewcommand{\thesubsection}{\Alph{subsection}}

\renewcommand{\thetheorem}{\Alph{subsection}.\arabic{theorem}}
\renewcommand{\theproposition}{\Alph{subsection}.\arabic{proposition}}
\renewcommand{\thelemma}{\Alph{subsection}.\arabic{lemma}}

\subsection{Algebraic geometry background}\label{appendix:alggeo}

Below we give a brief but comprehensive account of the algebraic geometry used in this paper. The content of the first part of the section is standard and can be found (often in much greater generality) in most textbooks on basic algebraic geometry, such as~\cite{harris_1995} or~\cite{shafarevich_2013}. The second part of the section describes the notion of generic points that is used in this paper. Its content is somewhat specialized, but is not new in the rigidity theory literature (see~\cite{garamvolgyi.etal_2022,gortler.etal_2019}, for example).

A set $X \subseteq \CC^k$ is an \emph{affine variety} if it is the set of common zeroes of some polynomials $f_i \in \CC[x_1,\ldots,x_k], i \in I$. The \emph{Zariski topology} on $\CC^k$ is the topology whose closed sets are the affine varieties of $\CC^k$.%
For an arbitrary set $X \subseteq \CC^k$, we use $\overline{X}$ to denote the \emph{Zariski closure} of $X$, which is the smallest affine variety containing $X$. A subset $U$ of an affine variety $X$ is a \emph{Zariski open subset} of $X$ if it can be written as $X - Y$ for some affine variety $Y$, and $U$ is \emph{Zariski dense in $X$} if $\overline{U} = X$.
An affine variety $X$ is \emph{homogeneous} if it is invariant under scaling, that is, $cx \in X$ for every $x \in X$ and $c \in \CC$.
An important class of homogeneous irreducible affine varieties are \emph{linear spaces}, which are just linear subspaces of $\CC^k$. 

An affine variety is \emph{irreducible} if it is nonempty and cannot be written as the proper union of two affine varieties. 
In an irreducible variety $X$, every nonempty Zariski open subset 
is Zariski dense in $X$. %
The \emph{dimension} of a nonempty affine variety $X$, denoted by $\dim(X)$, is the largest number $k$ for which there exists a chain of irreducible varieties $X_0 \subsetneq X_1 \subsetneq \ldots \subsetneq X_k \subseteq X$. It follows from this definition that if $X' \subsetneq X$ are irreducible varieties, then $\dim(X') < \dim(X)$. An affine variety is zero-dimensional if and only if it is finite.

A subset $X \subseteq \CC^k$ is \emph{Zariski constructible} if it can be obtained from a finite number of affine varieties by taking intersections, unions and complements. 
A Zariski constructible set $X$ is \emph{irreducible} if its Zariski closure is irreducible. Every Zariski constructible set contains a dense Zariski open subset of its closure.
Let $X \subseteq \CC^k$ be an irreducible Zariski constructible set.
A function $f: X \rightarrow \CC^\ell$ is a \emph{polynomial map} if there exist polynomials $f_1,\ldots,f_\ell \in \CC[x_1,\ldots,x_k]$ such that $f(x) = (f_1(x),\ldots,f_\ell(x))$ for all $x \in X$. Polynomial maps are continuous with respect to the Zariski topology, and it follows from topological considerations that if $X$ is an irreducible Zariski constructible set and $f : X \rightarrow \CC^\ell$ is a polynomial map, then $\overline{f(X)}$ is also irreducible.

Let $X \subseteq \CC^k$ be an irreducible affine variety, and let $I(X) \subseteq \CC[x_1,\ldots,x_k]$ be the ideal of polynomials vanishing on $X$. Fix a generating set $\{f_1,\ldots,f_\ell\}$ of $I(X)$, and consider the polynomial map $f = (f_1,\ldots,f_\ell) : \CC^k \rightarrow \CC^\ell$. We let $df$ denote the \emph{Jacobian} of $f$, which we view as a $k \times \ell$ matrix of polynomials. At each point $x \in X$, we define the \emph{tangent space of $X$ at $x$} (denoted by $T_x(X)$)  as the kernel of $df_x$, the Jacobian of $f$ evaluated at $x$. This is a linear subspace of $\CC^k$, and it is independent of the choice of generating polynomials of $I(X)$. The dimension of $T_x(X)$ is always at least the dimension of $X$. If the two numbers are equal, then $x$ is a \emph{smooth point} of $X$; otherwise it is a \emph{singular point}. The singular points of $X$ form a proper subvariety called the \emph{singular locus} of $X$. If $X'$ is an irreducible affine variety with $X' \subseteq X$ and $x \in X'$, then $T_x(X')$ is a linear subspace of $T_x(X)$.

Let $f : X \rightarrow \CC^\ell$ be a polynomial map and suppose that $f(X) \subseteq Y$ for some irreducible affine variety $Y$. Fix a point $x \in X$, and let $y = f(x) \in Y$. Then we have $df_x(T_x(X)) \subseteq T_y(Y)$. %
The following result, sometimes called Bertini's theorem, is an algebraic analog of Sard's lemma from differential topology.

\begin{theorem}\label{theorem:bertini} (\cite[Section 6.2]{shafarevich_2013})
    Let $X \subseteq \CC^k$ be an irreducible affine variety, let $f : X \rightarrow \CC^\ell$ be a polynomial map, and let $Y = \overline{f(X)}$. For a nonempty Zariski open subset of points $x \in X$, the map $df_x : T_x(X) \rightarrow T_{f(x)}(Y)$ is surjective. 
\end{theorem}

Let $f: X \rightarrow \CC^\ell$ be a polynomial map and let $y \in f(X)$. We call $f^{-1}(y)$ the  \emph{fiber of $f$ over $y$} in $X$. This is also an affine variety. %

\begin{theorem}\label{theorem:fiberdimension}\cite[Theorem 1.25]{shafarevich_2013} \emph{(Fiber dimension theorem.)}\newline
    Let $X \subseteq \CC^k$ be an irreducible affine variety, let $f : X \rightarrow \CC^\ell$ be a polynomial map, and let $Y = \overline{f(X)}$. Then for a nonempty Zariski open subset of points $x \in X$, the fiber of $f$ over $f(x)$ is nonempty and has dimension $\dim(X) - \dim(Y)$. %
\end{theorem}

An affine variety $X \subseteq \CC^k$ is \emph{defined over $\QQ$} if it is the set of common zeroes of some polynomials with rational coefficients. This is equivalent to requiring that the ideal $I(X)$ of polynomials vanishing on $X$ has a generating set consisting of polynomials with rational coefficients.  More generally, a Zariski constructible set $X \subseteq \CC^k$ is \emph{defined over $\QQ$} if it can be obtained  from a finite number of affine varieties defined over $\QQ$ by taking intersections, unions and complements. The first part of the following proposition is (a special case of) Chevalley's theorem.

\begin{proposition}\label{theorem:constructiblesetsoverQ}
    Let $X \subseteq \CC^k$ be a Zariski constructible set defined over $\QQ$, and let $f: X \rightarrow \CC^\ell$ be a polynomial map with rational coefficients. 
    \begin{enumerate}
        \item (\cite[Theorem 1.22]{basu.etal_2006}) The image $f(X)$ is also Zariski constructible and defined over $\QQ$.
        \item (\cite[Theorem A.8]{gkioulekas.etal_2020}) The variety $\overline{X}$ is also defined over $\QQ$.
        \item (\cite[Lemma A.4]{gortler.etal_2019}) There is a proper subvariety $X'$ of $\overline{X}$ defined over $\QQ$ such that $\overline{X} - X' \subseteq X$.
    \end{enumerate}
\end{proposition} 

Let $X \subseteq \CC^k$ be an irreducible variety defined over $\QQ$. We say that a point $x \in \CC^k$ is \emph{$X$-generic} if $x \in X$ and the only polynomials with rational coefficients that vanish at $x$ are those that vanish at every point of $X$. In other words, $x$ is $X$-generic if it is not contained in any proper subvariety of $X$ that is defined over $\QQ$. More generally, a point of an irreducible Zariski constructible set $X$ is \emph{$X$-generic} if it is $\overline{X}$-generic. Note that by the \cref{theorem:constructiblesetsoverQ}(c), every $\overline{X}$-generic point lies in $X$. The following lemmas describe some basic properties of this notion of genericity.%

\begin{lemma}\label{lemma:genericdense}\cite[Lemma A.6]{gortler.etal_2019} \emph{(Density of generic points.)}\newline
    Let $X' \subseteq X \subseteq \CC^k$ be irreducible affine varieties with $X$ defined over $\QQ$. If $X'$ contains an $X$-generic point, then the set of $X$-generic points in $X'$ is Zariski dense in $X'$.
\end{lemma}

\begin{lemma}\label{lemma:genericimage} \emph{(Generic points and polynomial maps.)}\newline
    Let $X \subseteq \CC^k$ be an irreducible affine variety defined over $\QQ$, let $f : X \rightarrow \CC^\ell$ be a polynomial map with rational coefficients, and let $Y = \overline{f(X)} \subseteq \CC^\ell$.
    \begin{enumerate}
        \item (\cite[Lemma A.7]{gortler.etal_2019}) If $x \in X$ is $X$-generic, then $f(x)$ is $Y$-generic.
        \item (\cite[Lemmas 4.4 and A.8]{gortler.etal_2019}) If $y \in Y$ is $Y$-generic, then the fiber $f^{-1}(y)$ contains an $X$-generic point. In particular, every $Y$-generic point lies in $f(X)$.
    \end{enumerate}
\end{lemma}

Note that in the preceding statement, $Y$ is irreducible, and it is defined over $\QQ$ by \cref{theorem:constructiblesetsoverQ}, so it is reasonable to talk about $Y$-generic points.

\begin{lemma}\label{lemma:genericproperties} \emph{(Properties of generic points.)}\newline
    Let $X \subseteq \CC^k$ be an irreducible affine variety defined over $\QQ$, let $x \in X$ be an $X$-generic point, and let $f: X \rightarrow \CC^\ell$ be a polynomial map with rational coefficients. Finally, let $Y = \overline{f(X)}$. We have the following.
    \begin{enumerate}
        \item $x$ is a smooth point.
        \item The fiber of $f$ over $f(x)$ has dimension $\dim(X) - \dim(Y)$.
        \item The map $df_x : T_x(X) \rightarrow T_{f(x)}(Y)$ is surjective.
    \end{enumerate}
\end{lemma}

\noindent For more details about varieties over $\QQ$ and their generic points, see~\cite[Section 2.3.8]{garamvolgyi_2024}.

\subsection{Proofs for \texorpdfstring{\cref{section:preliminaries,section:genericstressmatroid}}{Sections 2 and 3}}\label{appendix:prelims}

\begin{proof}[Proof of \cref{lemma:largenullity}]
    \textit{(a)} We have $\stressnullity{G} = n$ if and only if for any generic realization $(G,p)$, $K_1(G,p) = \CC^n$. If $G$ is $\mathcal{R}_d$-independent, then $(G,p)$ only has the zero stress, so  $K_1(G,p) = \CC^n$ indeed holds. Otherwise, $(G,p)$ has a nonzero stress $\omega$. We can find some one-dimensional realization of $G$ that does not have $\omega$ as a stress by suitably perturbing the coordinates of $(G,p)$, so in this case $K_1(G,p) \neq \CC^n$.
    
    \textit{(b)} For $G = K_{d+2}$, it follows from \cref{theorem:ght} that $\stressnullity{G} = d+1 = n - 1$. Sufficiency follows from the observation that adding an isolated vertex increases $\stressnullity{G}$ by one, while adding $\mathcal{R}_d$-bridges does not change it (see \cref{lemma:bridgeconverse}).
    
    For necessity, suppose that $\stressnullity{G} = n-1$; by part \textit{(a)}, this implies that $G$ is not $\mathcal{R}_d$-independent. 
    Let $(G,p)$ be a generic realization of $G$ in $\CC^d$ and let $\omega$ be a nonzero stress of $(G,p)$. %
    We now have \[n-1 = \dim(K_1(G,p)) \leq \dim (\stresskernel{1}{G}(\omega)) < n,\] and hence $K_1(G,p) = \stresskernel{1}{G}(\omega)$. This means that the stress matrix $\Omega$ associated to $\omega$ has an $(n-1)$-dimensional kernel, and hence $\rk(\Omega) = 1$. It is not difficult to see that in this case $C$ must be a complete graph, see~\cite[Lemma 3.6]{garamvolgyi.jordan_2023}.
    
    For any $\mathcal{R}_d$-circuit $C$ of $G$ there is a stress of $(G,p)$ that is supported on $E(C)$, and thus by the previous paragraph, $C$ must be a copy of $K_{d+2}$, since this is the only complete $\mathcal{R}_d$-circuit. Moreover, there cannot be a different $\mathcal{R}_d$-circuit $C'$ in $G$, since then the union of $C$ and $C'$ would be noncomplete but still have a stress of $(G,p)$ that is supported on it, contradicting the argument in the previous paragraph. Hence there is a single copy $C$ of $K_{d+2}$ in $G$, and every edge $e \in E(G) - E(C)$ must be an $\mathcal{R}_d$-bridge in $G$, as claimed.
\end{proof}

\begin{proof}[Proof of \cref{lemma:ksum}]
    \textit{(a)} Let $E_0'$ denote the edge set of a minimally \rig{d} spanning subgraph of $G_1 \cap G_2$, and let us extend $E_0'$ to a basis $E_i'$ of $\mathcal{R}_d(G_i)$ for $i \in \{1,2\}$. It follows from \cref{lemma:whiteleygluing}(b) %
    that $E' = E_1' \cup E_2'$ is a basis of $\mathcal{R}_d(G)$. Each edge $uv \in E - E'$ is contained in an $\mathcal{R}_d$-circuit in $G_1$ or $G_2$ (or both), and hence there is a stress $\omega_{uv} \in S_1 \cup S_2$ that is nonzero on $uv$. Since $E'$ is a maximal independent subset in $\mathcal{R}_d(G)$, the set $\{\omega_{uv}, uv \in E - E'\}$ generates $S(G,p)$, and thus so does $S_1 \cup S_2$.
    
    \textit{(b)}
    This is folklore, see, e.g.,~\cite[Lemma 6.3]{bernstein.etal_2021}.

    \textit{(c)} Let $\{u,v\} = V_1 \cap V_2$. By a slight abuse of notation, let us identify each of $S_1,S_2,S(G_1+uv,p_1),S(G_2+uv,p_2)$ and $S(G,p)$ as a subspace of $S(G+uv,p)$. Using part \emph{(b)} we can deduce that  
    \begin{equation}\label{eq:stressdirectsum}
        S_1 \oplus S_2 \subseteq S(G,p) \subseteq S(G+uv,p) = S(G_1+uv,p_1) \oplus S(G_2+uv,p_2). 
    \end{equation}
    If $\{u,v\}$ is not $\mathcal{R}_d$-linked in either $G_1$ or $G_2$, then $S_i = S(G_i+uv,p_i)$ for $i \in \{1,2\}$, and hence the chain of containments above gives $S(G,p) = S_1 \oplus S_2$. If on the other hand $\{u,v\}$ is $\mathcal{R}_d$-linked in precisely one of $G_1$ and $G_2$, then the dimension of $S_1 \oplus S_2$ is one less than the dimension of $S(G_1+uv,p_1) \oplus S(G_2+uv,p_2)$, and similarly, the dimension of $S(G,p)$ is one less than the dimension of $S(G+uv,p)$. Again, the chain of containments in \cref{eq:stressdirectsum} implies that $S(G,p) = S_1 \oplus S_2$.
\end{proof}

In the following proofs we shall use a more general form of Chevalley's theorem (\cref{theorem:constructiblesetsoverQ}(a)). A \emph{formula in the language of fields with coefficients in $\CC$} is a first-order formula built from polynomial equations of the form ``$f(x) = 0$'' with $f \in \CC[x_1,\ldots,x_k]$ using the logical connectives ``and'' and ``or'', negation, and existential and universal quantifiers. We say that such a formula \emph{has coefficients in $\QQ$} if the polynomials appearing in it all have rational coefficients. A stronger form of Chevalley's theorem states that any subset of $\CC^k$ that can be defined using a formula in the language of fields is constructible, and if the formula has coefficients in $\QQ$, then the resulting constructible set is also defined over $\QQ$, see~\cite[Corollary 1.24]{basu.etal_2006}.

\begin{proof}[Proof sketch of \cref{lemma:dualoverQ}]
    The union in \cref{eq:dualdescription} can be defined by a formula in the language of fields and, since $X$ is defined over $\QQ$, this formula has coefficients in $\QQ$. By Chevalley's theorem, this set is constructible and defined over $\QQ$, so its closure $C_X$ is also defined over $\QQ$ by \cref{theorem:constructiblesetsoverQ}(b). Hence by \cref{theorem:constructiblesetsoverQ}(a) and (b), $X^*$ is also defined over $\QQ$. 
    
    It follows now from \cref{lemma:genericimage} that for every $X^*$-generic point $\omega$, there is a $C_X$-generic pair $(x,\omega)$, and that $x$ is $X$-generic for such a pair. Moreover, $x$ is a smooth point of $X$ by \cref{lemma:genericproperties}(a), and thus by \cref{eq:conormalprojection}, we have $\omega' \in {T_x(X)}^\perp$ whenever $(x,\omega') \in C_X$. It follows from these observations that the union in \cref{eq:genericclosure} contains all $C_X$-generic points, and hence it is Zariski dense in $C_X$ by \cref{lemma:genericdense}.  
\end{proof}

\begin{proof}[Proof of \cref{lemma:dualconnectivity}]
    By symmetry, it is enough to show the ``only if'' direction. Let us suppose that $(E_1,E_2)$ is a separation of $\algmat{X}$. Then by \cref{lemma:algmatseparation}, we have $X = X_{E_1} \times X_{E_2}$. It is known that in this case for every point $x = (x_1,x_2) \in X$, we have $T_x(X) = T_{x_1}(X_{E_1}) \times T_{x_2}(X_{E_2})$, and thus ${T_x(X)}^\perp = {T_{x_1}(X_{E_1})}^\perp \times {T_{x_2}(X_{E_2})}^\perp$. Combined with the  fact that $\dim(X) = \dim(X_{E_1}) + \dim(X_{E_2})$ this also shows that $x = (x_1,x_2)$ in $X$ is smooth if and only if $x_i$ is smooth in $X_{E_i}$ for $i \in \{1,2\}$. 
    
    Let us denote the smooth locus of $X$ and $X_i$ by $X_{sm}$ and ${(X_i)}_{sm}$, respectively.
    Recall that by \cref{eq:dualdescription}, $X^*$ is the Zariski closure of
    \begin{align*}
        \bigcup \{{T_{x}(X)}^\perp : x \in X_{sm}\}  
        &= \bigcup \left\{{T_{x_1}(X_{E_1})}^\perp \times {T_{x_2}(X_{E_2})}^\perp : x_i \in {(X_i)}_{sm}, i \in \{1,2\}\right\} \\
        &= \bigcup \left\{{T_{x_1}(X_{E_1})}^\perp : x_1 \in {(X_1)}_{sm}\right\} \times \bigcup \left\{{T_{x_2}(X_{E_2})}^\perp : x_2 \in {(X_2)}_{sm}\right\}.
    \end{align*}
    Using the fact that the Zariski closure of a Cartesian product of two sets is the Cartesian product of the respective Zariski closures, we obtain $X^* = {(X_{E_1})}^* \times {(X_{E_2})}^*$. This implies that ${(X^*)}_{E_i} = {(X_{E_i})}^*$ for $i \in \{1,2\}$, and now \cref{lemma:algmatseparation} shows that $(E_1,E_2)$ is a separation of $\algmat{X^*}$.
\end{proof}

\begin{proof}[Proof of \cref{lemma:contactlocusstresskernel}]
    Let $(G,p)$ be a generic framework in $\CC^d$ with $\omega \in S(G,p)$, and let $x = m_{d,G}(p)$. By \cref{lemma:genericimage}, $x$ is an $M_{d,G}$-generic point, and consequently it is a smooth point by \cref{lemma:genericproperties}(a). It follows from \cref{lemma:stresstangentspace} that $x \in \contactlocus{d}{G}(\omega) \subseteq M_{d,G}$. In particular, $\contactlocus{d}{G}(\omega)$ contains an $M_{d,G}$-generic point. Now \cref{lemma:genericdense} implies that the $M_{d,G}$-generic points are Zariski dense in $\contactlocus{d}{G}(\omega)$, and hence we can write
\[\contactlocus{d}{G}(\omega) = \overline{\{m_{d,G}(q) : (G,q) \text{ is generic and } q \in \stresskernel{d}{G}(\omega)\}}.\]
    Let us denote the set of generic configurations in $\stresskernel{d}{G}(\omega)$ by $\gen{\stresskernel{d}{G}(\omega)}$. Since $\stresskernel{d}{G}(\omega)$ is a linear space that contains at least one generic configuration (namely, $p$), \cref{lemma:genericdense} implies that $\stresskernel{d}{G}(\omega) = \overline{\gen{\stresskernel{d}{G}(\omega)}}$. It follows that
\[\contactlocus{d}{G}(\omega) = \overline{m_{d,G}(\gen{\stresskernel{d}{G}(\omega)})} = \overline{m_{d,G}(\overline{\gen{\stresskernel{d}{G}(\omega)}})} = \overline{m_{d,G}(\stresskernel{d}{G}(\omega))},\]where the second equality uses the continuity of $m_{d,G}$ in the Zariski topology. 
\end{proof}

\begin{proof}[Proof sketch of \cref{lemma:projectivetechnical}]
    We use the notation set before the statement of \cref{lemma:projectivetechnical}. We first show that $P(G,p,\omega)$ is an irreducible constructible set. The set $\mathcal{F} \subseteq \CC^{(d+1) \times (d+1)}$ of feasible matrices is a Zariski open subset of $\CC^{(d+1) \times (d+1)}$. By the definition of a projective image, $P(G,p,\omega)$ is the image of this open subset under a regular map, i.e., it can be written as $\Phi(\mathcal{F})$, where $\Phi(x) = (f_1(x)/g_1(x),\ldots,f_k(x)/g_k(x))$ with $f_i,g_i$ polynomials such that $g_i$ is nonvanishing over $\mathcal{F}$, for $i \in \{1,\ldots,k\}$. It is not difficult to see that $\Phi(\mathcal{F})$ can be defined by a formula in the language of fields, and hence Chevalley's theorem implies that $P(G,p,\omega)$ is a constructible set. Irreducibility follows from the irreducibility of the set $\mathcal{F}$ of feasible matrices and the fact that $\Phi$ is a regular map.
    
    Let $\pi_1$ denote the projection of $P(G,p,\omega) \subseteq (\CC^{d})^V \times \CC^E$ onto the first factor. By Chevalley's theorem, $\pi_1(P(G,p,\omega))$ is a constructible set, and it contains the generic configuration $p$. It follows from \cref{lemma:genericdense} that the generic configurations are Zariski dense in $\overline{\pi_1(P(G,p,\omega))}$, and since a constructible set contains a dense open set of its closure, the generic configurations are also dense in $\pi_1(P(G,p,\omega))$. This implies that for a Zariski dense subset of points $(q,\omega')$ in $P(G,p,\omega)$, $q$ is generic, and hence $\omega'$ is \quasigen. It follows that the \quasigen\ $d$-stresses are Zariski dense in $\pi_2(P(G,p,\omega))$. Since \quasigen\ $d$-stresses belong to $M_{d,G}^*$, by taking the Zariski closure, we deduce that $\pi_2(P(G,p,\omega))$ is also contained in $M_{d,G}^*$, as claimed.
\end{proof}

\subsection{Proofs for \texorpdfstring{\cref{section:globallylinked}}{Section 4}}\label{appendix:stresslinked}

\begin{proof}[Proof of \cref{lemma:gaussfiberchar}]
    We first determine the dimension of $L(G,p)$. Let $R_p = m_{d,G}^{-1}(m_{d,G}(p))$ denote the set of configurations $q$ for which $(G,q)$ is equivalent to $(G,p)$. By \cref{lemma:genericproperties}(b), we have $\dim(R_p) = dn - r_d(G)$. It follows from \cref{theorem:connelly} that $R_p \subseteq \sharedstresskernel{d}(G,p)$, so the restriction $m_{d,G} : \sharedstresskernel{d}(G,p) \rightarrow L(G,p)$ has the same fiber over $m_{d,G}(p)$. By another application of \cref{lemma:genericproperties}(b), we have \[\dim(L(G,p)) = \dim(\sharedstresskernel{d}(G,p)) - \dim(R_p) = d\stressnullity{G} - dn + r_d(G).\]
    
    With this formula in hand, the statement is easy to prove. If $\{u,v\}$ is not $\mathcal{R}_d$-linked, then $r_d(G+uv) = r_d(G) + 1$ by definition and $\stressnullity{G+uv} = \stressnullity{G}$ by \cref{lemma:bridgeconverse}, so in this case $\dim(L(G+uv,p)) > \dim(L(G,p))$. If $\{u,v\}$ is $\mathcal{R}_d$-linked, then $r_d(G+uv) = r_d(G)$ and hence $\dim(L(G+uv,p)) = \dim(L(G,p))$ holds if and only if $\stressnullity{G+uv} = \stressnullity{G}$, which in turn is equivalent to $\{u,v\}$ being \strongly{d}.
\end{proof}

\begin{proof}[Proof of \cref{lemma:stronglylinkedsubgraph}]
    It is clear that if $\{u,v\}$ is $\mathcal{R}_d$-linked in $G_0$, then it is also $\mathcal{R}_d$-linked in $G$. Let $(G,p)$ be a generic framework in $\CC^d$, let $(G_0,p_0)$ be the corresponding subframework, and let $\omega$ be a stress of $(G+uv,p)$ that is nonzero on $uv$ and whose support is contained in $E(G_0+uv)$. (Such a stress exists since $\{u,v\}$ is $\mathcal{R}_d$-linked in $G_0$.) By \cref{lemma:def}(c), it suffices to show that for every $q \in \sharedstresskernel{d}(G,p)$, $\omega$ is also a stress of $(G+uv,q)$. Let $(G_0,q_0)$ be the subframework of $(G,q)$ corresponding to $G_0$. Since every stress of $(G_0,p_0)$ can also be viewed as a stress of $(G,p)$, and $(G,q)$ satisfies these stresses, we have $q_0 \in \sharedstresskernel{d}(G_0,p_0)$. By definition, the restriction $\widetilde{\omega}$ of $\omega$ to the edges of $G_0+uv$ is  a stress of $(G_0+uv,p_0)$. Since $\{u,v\}$ is \strongly{d} in $G_0$, it follows that $\widetilde{\omega}$ is also a stress of $(G_0+uv,q_0)$, and thus $\omega$ is also a stress of $(G+uv,q)$, as required.
\end{proof}

\begin{proof}[Proof of \cref{lemma:completeseparator}]
    If $\{u,v\}$ is \strongly{d} in $G_1$, then it is \strongly{d} in $G$ by \cref{lemma:stronglylinkedsubgraph}. Hence let us assume that $\{u,v\}$ is \strongly{d} in $G$. Let $(G,p)$ be a generic realization in $\CC^d$, and let $(G_i,p_i)$ be the subframeworks corresponding to $G_i$ for $i \in \{1,2\}$. The case when $G_1$ and $G_2$ are disjoint is easy to verify, so let us suppose that $G_1 \cap G_2$ is a (nonempty) complete graph.
    
    Since the pair $\{u,v\}$ is \strongly{d} in $G$, it is $\mathcal{R}_d$-linked in $G$. We claim that it is also $\mathcal{R}_d$-linked in $G_1$. This can be seen geometrically: indeed, if there were an infinite number of realizations $\{(G_1,q_i), i \in I\}$ all equivalent to $(G_1,p_1)$ and such that the edge lengths $\{m_{uv}(q_i), i \in I\}$ are pairwise different, then by gluing $(G_2,p_2)$ to these realizations (which we can do, since $G_1 \cap G_2$ is a complete graph and thus \globrig{d}) we would also obtain that $\{u,v\}$ is not $\mathcal{R}_d$-linked in $G$, a contradiction. It follows that there is a stress $\omega$ of $(G_1+uv,p_1)$ that is nonzero on $uv$. By \cref{lemma:def}(c), it suffices to show that for every $q_1 \in \sharedstresskernel{d}(G_1,p_1)$, $\omega$ is also a stress of $(G_1+uv,q_1)$. To this end, let us fix a configuration $q_1 \in \sharedstresskernel{d}(G_1,p_1)$. Let $S$ denote $V_1 \cap V_2$, and let us consider the subframeworks $(G[S],p_S)$ and $(G[S],q_S)$ obtained by restricting $(G_1,p_1)$ and $(G_1,q_1)$, respectively, to $G[S]$. The fact that $q_1 \in \sharedstresskernel{d}(G_1,p_1)$ implies $q_S \in \sharedstresskernel{d}(G[S],p_S)$, and since $G[S]$ is \globrig{d}, by \cref{theorem:ght}(c) we have that $q_S \in \affine{G[S]}{p_S}$. 
    
    Let $A$ be an affine transformation that sends $p_S$ to $q_S$, and let $q_2$ be the configuration obtained by applying $A$ to each point in the configuration $p_2$. Note that $S(G_2,p_2) \subseteq S(G_2,q_2)$. Consider the framework $(G,q)$ obtained by gluing $(G_1,q_1)$ and $(G_2,q_2)$. By a slight abuse of notation, we may view $S(G_i,p_i)$ and $S(G_i,q_i)$ as subsets of $S(G,p)$ and $S(G,q)$, respectively. We now have \[S(G,q) \supseteq S(G_1,q_1) \cup S(G_2,q_2) \supseteq S(G_1,p_1) \cup S(G_2,p_2),\] and by \cref{lemma:ksum}(a), $S(G,p)$ is generated by $S(G_1,p_1) \cup S(G_2,p_2)$. It follows that $S(G,p) \subseteq S(G,q)$, or in other words, that $q \in \sharedstresskernel{d}(G,p)$. By the assumption that $\{u,v\}$ is \strongly{d} in $G$, we have that $\omega$ is also a stress of $(G+uv,q)$. But since the support of $\omega$ is contained in $E(G_1)$, this means that it is a stress of $(G_1+uv,q_1)$, as desired.
\end{proof}

\end{document}